\documentclass[12pt]{amsart}
\usepackage{ amsmath,amssymb,amsthm,color, euscript, enumerate,longtable}
\usepackage{dsfont,comment}

\newcommand{\1}{ \mathds{1}}

\newcommand{\maru}[1]{{\ooalign{\hfil#1\/\hfil\crcr
\raise.167ex\hbox{\mathhexbox20D}}}}

\newcommand{\ruby}[2]{%
 \leavevmode
 \setbox0=\hbox{#1}%
 \setbox1=\hbox{\tiny #2}%
 \ifdim\wd0>\wd1 \dimen0=\wd0 \end{lemma}se \dimen0=\wd1 \fi
 \hbox{%
   \kanjiskip=0pt plus 2fil
   \xkanjiskip=0pt plus 2fil
   \vbox{%
     \hbox to \dimen0{%
       \tiny \hfil#2\hfil}%
     \nointerlineskip
     \hbox to \dimen0{\mathstrut\hfil#1\hfil}}}}

\DeclareMathOperator*{\fusion}{\boxtimes}

\newcommand{\Z}{\mathbb{Z}}
\newcommand{\C}{\mathbb{C}}
\newcommand{\R}{\mathbb{R}}

\newcommand{\Q}{\mathbb{Q}}

\newcommand{\g}{\mathfrak{g}}

\newcommand{\Fg}{\mathfrak{g}}

\newcommand{\h}{\mathfrak{h}}

\newcommand{\End}{\mathrm{End}}

\newcommand{\Aut}{\mathrm{Aut}\,}

\makeatletter \@addtoreset{equation}{section}

\theoremstyle{plain}
\newtheorem{theorem}{Theorem}[section]
\newtheorem{proposition}[theorem]{Proposition}
\newtheorem{lemma}[theorem]{Lemma}
\newtheorem{corollary}[theorem]{Corollary}

\theoremstyle{definition}

\theoremstyle{remark}
\newtheorem{remark}[theorem]{Remark}
\numberwithin{equation}{section}

\title[Holomorphic vertex operator algebras]{Reverse orbifold construction and uniqueness of holomorphic vertex operator algebras}
\author{Ching Hung Lam} %
  \address[C. H. Lam] {Institute of Mathematics, Academia Sinica, Taipei 10617, Taiwan and National Center for Theoretical Sciences of  Taiwan.}
  \email{chlam@math.sinica.edu.tw}
\author[H. Shimakura]{Hiroki Shimakura}%
\address[H. Shimakura]{Graduate School of Information Sciences,
Tohoku University,
Sendai 980-8579, Japan }%
\email {shimakura@tohoku.ac.jp}%
\date{}
\thanks{C.\,H. Lam was partially supported by MoST grant 104-2115-M-001-004-MY3 of Taiwan}
\thanks{H.\ Shimakura was partially supported by JSPS KAKENHI Grant Numbers 26800001 and 17K05154}
\thanks{C.\,H. Lam and H.\ Shimakura were partially supported by JSPS Program for Advancing Strategic International Networks to Accelerate the Circulation of Talented Researchers ``Development of Concentrated Mathematical Center Linking to Wisdom of the Next Generation".}

\newcommand{\sfr}[2]{\leavevmode\kern-.1em
  \raise.5ex\hbox{\the\scriptfont0 #1}\kern-.1em
  /\kern-.15em\lower.25ex\hbox{\the\scriptfont0 #2}}

\begin{document}

\begin{abstract}
In this article, we develop a general technique for proving the uniqueness of holomorphic vertex operator algebras based on the orbifold construction and its ``reverse" process. 
As an application, we prove that the structure of a strongly regular holomorphic vertex operator algebra of central charge $24$ is uniquely determined by its weight one Lie algebra if the Lie algebra has the type $E_{6,3}G_{2,1}^3$, $A_{2,3}^6$ or $A_{5,3}D_{4,3}A_{1,1}^3$.
\end{abstract}
\maketitle


\section{Introduction}

This article is a continuation of our program on classification of holomorphic vertex operator algebras (VOAs) of central charge $24$. In 1993,  Schellekens \cite{Sc93} obtained  a partial classification by
determining possible Lie algebra structures  for the weight one subspaces  of holomorphic VOAs of  central charge $24$. 
There are
$71$ cases in his list but not all  cases were constructed  explicitly at that time. Recently, there is much progress towards the classification; all $71$ cases have been constructed (\cite{FLM,DGM,Lam,LS,LS3,Mi3,SS,EMS,LS4,LLin}) and Schellekens' list has also been reproved in \cite{EMS}. The main technique for construction is often referred as to  ``Orbifold construction". 
When the central charge is $24$, it is also conjectured that the Lie
algebra structure of the weight one subspace will determine the holomorphic VOA structure uniquely.
This would be an analogue of the uniqueness of the Niemeier lattices: the isometry class of a positive-definite even unimodular lattice of rank $24$ is uniquely determined by the root system consisting of the norm $2$ vectors. 
However, this conjecture has been proved only for the following $26$ cases.\footnote{After we 
finished this work, except for the case $V_1=0$, the uniqueness for all the other cases has been proved  by using the main 
idea of this article \cite{EMS2,KLL,LLin,LS18,LS5}. }  
In \cite{DM}, it was shown that if the weight one Lie algebra has Lie rank $24$, then the holomorphic VOA is isomorphic to the lattice VOA associated with one of the $24$ Niemeier lattices.
For the other cases, the classification of holomorphic framed VOAs of central charge $24$ in \cite{LS2} implies that if the weight one Lie algebra has the type $A_{1,2}^{16}$ or $E_{8,2}B_{8,1}$, then the holomorphic VOA structure of central charge $24$ is unique.  Indeed, in these two cases, the subVOA generated by the weight one Lie algebra is framed, and so is the corresponding holomorphic VOA.

In this article, we discuss a general technique for proving the uniqueness of holomorphic vertex operator algebras. 
The key idea is to reverse the orbifold construction based on the following simple observation: 

Let $V$ be a (strongly regular) holomorphic VOA and let $\sigma$ be an order $n$ automorphism of $V$.
Assume that $\sigma$ satisfies certain condition (see Condition (I) in Section \ref{S:Orb}) so that one can apply the $\Z_n$-orbifold construction to $V$ and $\sigma$ (cf. \cite{EMS,MoPhD}).
Then the resulting holomorphic VOA $\widetilde{V}_\sigma$ is a $\Z_n$-graded simple current extension of the $\sigma$-fixed-point subalgebra $V^\sigma$.
The $\Z_n$-grading of $\widetilde{V}_\sigma$ gives rise to an order $n$ automorphism $\tau$ of $\widetilde{V}_\sigma$.
Clearly, $(\widetilde{V}_\sigma)^\tau=V^\sigma$.
In addition, the $\Z_n$-orbifold construction associated with $\widetilde{V}_\sigma$ and $\tau$ gives the original holomorphic VOA $V$ as a $\Z_n$-graded simple current extension of $(\widetilde{V}_\sigma)^\tau$, that is, $\widetilde{(\widetilde{V}_\sigma)}_\tau\cong V$ (see Corollary \ref{C:RevO}).
We call this procedure the \emph{reverse orbifold construction}, which is called the inverse orbifold in \cite{EMS,MoPhD}.

Now, we will explain our hypotheses for the uniqueness results.
Let $\g$ be a Lie algebra, $\mathfrak{p}$ a subalgebra of $\g$.
Let $n$ be a positive integer and let $W$ be a strongly regular holomorphic VOA of central charge $c$.
\begin{enumerate}[(i)]
\item For any holomorphic VOA $V$ of central charge $c$ with $V_1\cong\g$, there exists an order $n$ automorphism $\sigma$ of $V$ such that $\widetilde{V}_{\sigma}\cong W$ and $V_1^{\sigma}\cong\mathfrak{p}$.
\item Any order $n$ automorphism $\varphi$ of $W$ belongs to a unique conjugacy class if $(W^\varphi)_1\cong\mathfrak{p}$ and $(\widetilde{W}_\varphi)_1\cong\g$.  
\end{enumerate}

Under the above hypotheses, it is easy to show that any holomorphic 
VOA $V$ of central charge $c$ with $V_1\cong \g$ is isomorphic to $\widetilde{W}_\varphi$ (see 
Section \ref{sec:4.2}).
Indeed, Hypothesis (i) implies $\widetilde{V}_{\sigma}\cong W$, and the reverse orbifold construction shows $\widetilde{W}_\tau\cong V$.
In addition, Hypothesis (ii) implies that $\tau$ and $\varphi$ are conjugate, and hence $V\cong \widetilde{W}_\tau\cong\widetilde{W}_\varphi$.
In particular, the holomorphic VOA structure of central charge $c$ with weight one Lie algebra $\g$ is unique.
The similar arguments have also been used in \cite{LY07,LS2} for proving the uniqueness of certain framed VOAs.
As an application, we prove the following main theorem:

\begin{theorem}\label{Thm:main} The structure of a strongly regular holomorphic vertex operator algebra of central charge $24$ is uniquely determined by its weight one Lie algebra if the Lie algebra has the type $E_{6,3}G_{2,1}^3$, $A_{2,3}^6$ or $A_{5,3}D_{4,3}A_{1,1}^3$.
\end{theorem}

One of the reasons why we deal with the three Lie algebras above is that the corresponding holomorphic VOAs of central charge $24$ were explicitly constructed in \cite{Mi3,SS} by applying the $\Z_3$-orbifold construction to suitable Niemeier lattice VOAs and order $3$ automorphisms.
Hence we can easily find automorphisms for the reverse orbifold constructions, and, indeed, they are inner.
There is also a dimension formula for the $\Z_3$-orbifold construction (see Section 
\ref{S:Orb}), which plays important roles in verifying Hypothesis (i).

\medskip

The organization of the article is as follows.
In Section 3, we prove the uniqueness of the conjugacy classes of certain order $3$ automorphisms of the lattice VOAs associated with Niemeier lattices with root sublattices of type $E_6^4$ and $D_4^6$, which will prove Hypothesis (ii).
The argument is based on the results in \cite{Kac} about conjugacy classes of automorphisms of finite order of simple Lie algebras and the structure of the automorphism group of a lattice VOA (\cite{FLM,DN}).
In Section 4, we review the general $\Z_n$-orbifold construction from \cite{EMS,MoPhD}.
We also recall a dimension formula about weight one Lie algebras for $\Z_3$-orbifold construction from \cite{Mo,MoPhD}. This formula 
will be useful in determining the Lie 
algebra structure of the resulting holomorphic VOA. In addition, we establish the main 
technique for the uniqueness of a holomorphic VOA using the reverse orbifold construction.
In Sections 5, 6 and 7, we prove the uniqueness of holomorphic VOAs.
For each case, we verify Hypothesis (i) using the results established in Sections 3 and 4.
\begin{center}
{\bf Notations}
\begin{small}
\begin{longtable}{ll}\\
$(\cdot|\cdot)$& positive-definite symmetric bilinear form of a lattice, or \\
& the normalized Killing form so that $(\alpha|\alpha)=2$ for long roots $\alpha$.\\
$\langle\cdot|\cdot\rangle$& the normalized symmetric invariant bilinear form on a VOA $V$\\
& so that $\langle \1|\1\rangle=-1$, equivalently, $\langle a|b\rangle\1=a_{(1)}b$ for $a,b\in V_1$.\\
$\mathfrak{g}^g$& the fixed-point subalgebra of $g$ in a Lie algebra $\mathfrak{g}$.\\
$L_\mathfrak{g}(k,0)$& the simple affine VOA associated with simple Lie algebra $\mathfrak{g}$ at level $k$.\\
$L_\Fg(k,\lambda)$& the irreducible $L_\Fg(k,0)$-module with highest weight $\lambda$.\\
$M^{(h)}$& the $\sigma_h$-twisted $V$-module constructed from a $V$-module $M$ by Li's $\Delta$-operator.\\
$Ni(R)$& a Niemeier lattice with root sublattice $R$.\\
$\Pi(\lambda)$& the set of all weights of the irreducible module with highest weight $\lambda$\\
& over a simple Lie algebra.\\
$\sigma_h$& the inner automorphism $\exp(-2\pi\sqrt{-1}h_{(0)})$ of a VOA $V$ associated with $h\in V_1$.\\
$U(1)$& a $1$-dimensional abelian Lie algebra.\\
$V^{g}$& the set of fixed-points of an automorphism $g$ of a VOA $V$.\\
$X_n$& (the type of) a root system, a simple Lie algebra or a root lattice.\\
$X_{n,k}$& (the type of) a simple Lie algebra whose type is $X_n$ and level is $k$.\\
\end{longtable}
\end{small}
\end{center}

\section{Preliminary}
In this section, we will review some fundamental results about VOAs. Our notations for VOAs  and 
their modules are standard (see \cite{Bo,FLM,FHL}). 
In particular, we use $Y(a,z)=\sum_{i\in\Z}a_{(i)}z^{-i-1}$ (resp. $\1$ and $\omega$) to denote the 
vertex operator associated with $a\in V$ (resp. the vacuum vector and the conformal vector).
For $a\in V$ and $n\in\Z$, we call $a_{(n)}$ the {\it $n$-th mode} of $a$.
A  vertex operator subalgebra (or a {\it subVOA}) of $V$ is said to be a full subVOA  
if it has the same conformal vector as $V$.

For $g\in\Aut (V)$ of order $n$, a $g$-twisted $V$-module $(M,Y_M)$ is a $\C$-graded vector space 
$M=\bigoplus_{m\in\C} M_{m}$ equipped with a linear map
$$Y_M(a,z)=\sum_{i\in(1/n)\Z}a_{(i)}z^{-i-1}\in (\End M)[[z^{1/n},z^{-1/n}]],\quad a\in V$$
satisfying a number of conditions (\cite{DLM2}). We also assume that $L(0)v=kv$ if $v\in M_k$. 
If $M$ is irreducible, then there exists $w\in\C$ such that $M=\bigoplus_{m\in(1/n)\Z_{\geq 
0}}M_{w+m}$ and $M_w\neq0$.
The number $w$ is called the \emph{lowest $L(0)$-weight} of $M$.

A VOA is said to be  {\it rational} if its admissible module category is semisimple.
A rational VOA is said to be {\it holomorphic} if it itself is the only irreducible module up
to isomorphism.
A VOA $V=\bigoplus_{n\in \Z} V_n$ is said to be {\it of CFT-type} if $V_0=\C\1$ (and hence $V_n=0$ for $n<0$),  and is said to 
be {\it $C_2$-cofinite} if the codimension in $V$ of the subspace spanned by the vectors of form 
$u_{(-2)}v$, $u,v\in V$  is finite.
A module is said to be {\it self-dual} if its contragredient module is isomorphic to itself.
A VOA is said to be {\it strongly regular} if it is rational, $C_2$-cofinite, self-dual and of 
CFT-type.

\medskip

\subsection{Weight one Lie algebras} Let $V$ be a VOA of CFT-type.
Then, the weight one subspace $V_1$ has a Lie algebra structure via the $0$-th mode. Moreover, the 
$n$-th modes
$v_{(n)}$, $v\in V_1$, $n\in\Z$, define  an affine representation of the Lie algebra $V_1$ on $V$.
For a simple Lie subalgebra $\mathfrak{s}$ of $V_1$, the {\it level} of $\mathfrak{s}$ is defined to 
be the scalar by which the canonical central element acts on $V$ as the affine representation.
When the type of the root system of $\mathfrak{s}$ is $X_n$ and the level of $\mathfrak{s}$ is $k$, 
we denote by $X_{n,k}$ the type of $\mathfrak{s}$.

Assume that $V$ is self-dual.
Then there exists a symmetric invariant bilinear form $\langle\cdot|\cdot\rangle$ on $V$, which is 
unique up to scalar (\cite{Li3}).
We normalize it so that  $\langle\1|\1\rangle=-1$.
Then for $a,b\in V_1$, we have $\langle a|b\rangle\1=a_{(1)}b$.
For an element $a\in V_1$, $\exp(a_{(0)})$ is an automorphism of $V$.
The subgroup  of $\Aut(V)$  generated by these automorphisms is called the {\it inner automorphism group} and its elements are called {\it inner automorphisms.}
For a semisimple element $h\in V_1$, we often consider the inner automorphism 
$\sigma_h=\exp(-2\pi\sqrt{-1}h_{(0)})$ associated with $h$.

Assume that $V_1$ is semisimple.
Let $\mathfrak{H}$ be a Cartan subalgebra of $V_1$ and $(\cdot|\cdot)$ the Killing form on $V_1$.
We identify the dual $\mathfrak{H}^*$ with $\mathfrak{H}$ via $(\cdot|\cdot)$ and normalize 
$(\cdot|\cdot)$ so that $(\alpha|\alpha)=2$ for any long root $\alpha\in\mathfrak{H}$.
In this article, {\it weights} for $\mathfrak{H}$ are defined via $(\cdot|\cdot)$, that is, the 
weight of a vector $v\in V$ for $\mathfrak{H}$ is $\lambda\in\mathfrak{H}$ if 
$x_{(0)}v=(x|\lambda)v$ for all $x\in\mathfrak{H}$.
Remark that for $h\in\mathfrak{H}$, $\sigma_h$ acts on a vector with the weight $\lambda$ as the 
scalar multiple by $\exp(-2\pi\sqrt{-1}(h|\lambda))$.
The following lemma is immediate from the commutator relations of $n$-th modes (cf.\ 
{\cite[(3.2)]{DM06}}).

\begin{proposition}\label{Prop:level} 
Let $\mathfrak{s}$ be a simple Lie subalgebra of $V_1$ with level $k$.
Then $\langle\cdot|\cdot\rangle=k(\cdot|\cdot)$ on $\mathfrak{s}$.
\end{proposition}

Next we recall some results related to the Lie algebra $V_1$.

\begin{proposition}[{\cite[Theorem 1.1, Corollary 4.3]{DM06}}]\label{Prop:posl} Let $V$ be a 
strongly regular VOA.
Then $V_1$ is reductive.
Let $\mathfrak{s}$ be a simple Lie subalgebra of $V_1$.
Then $V$ is an integrable module for the affine representation of $\mathfrak{s}$ on $V$, and the 
subVOA generated by $\mathfrak{s}$ is isomorphic to the simple affine VOA associated with 
$\mathfrak{s}$ at positive integral level.
\end{proposition}

\begin{proposition} [{\cite[(1.1), Theorem 3 and Proposition 4.1]{DMb}}]\label{Prop:V1} Let $V$ be a 
strongly regular holomorphic VOA of central charge $24$.
If the Lie algebra $V_1$ is neither $\{0\}$ nor abelian, then $V_1$ is semisimple, and the conformal 
vectors of $V$ and the subVOA generated by $V_1$ are the same.
In addition, for any simple ideal of $V_1$ at level $k$, the identity$$\frac{h^\vee}{k}=\frac{\dim 
V_1-24}{24}$$
holds, where $h^\vee$ is the dual Coxeter number.
\end{proposition}

\begin{proposition}[{\cite[Corollary 1.4]{DM}}]\label{P:NVOA} Let $V$ be a strongly regular 
holomorphic VOA of central charge $24$.
If the Lie rank of $V_1$ is $24$, then $V$ is isomorphic to a Niemeier lattice VOA.
\end{proposition}

\subsection{$\Delta$-operator, simple affine VOAs and twisted modules}

In this subsection, we recall the construction of certain twisted modules given by Li \cite{Li}.
We also discuss the lowest $L(0)$-weight of such a twisted module over a simple affine VOA.

Let $V$ be a VOA of CFT-type.
Let $h\in V_1$ such that $L(1)h=0$ and $h_{(0)}$ acts semisimply on $V$.
Note that if $V$ is self-dual, then $L(1)V_1=0$.
Suppose that there exists a positive integer $n\in\Z_{>0}$ such that spectra of $h_{(0)}$ on $V$ 
belong to $({1}/{n})\Z$.
Then $\sigma_h=\exp(-2\pi\sqrt{-1}h_{(0)})$ is an automorphism of $V$ with $\sigma_h^n=1$.

Let $\Delta(h,z)$ be Li's $\Delta$-operator defined in \cite{Li}, i.e.,
\[
\Delta(h, z) = z^{h_{(0)}} \exp\left( \sum_{i=1}^\infty \frac{h_{(i)}}{-i} (-z)^{-i}\right).
\]

\begin{proposition}[{\cite[Proposition 5.4]{Li}}]\label{Prop:twist}
Let $\sigma$ be an automorphism of $V$ of finite order and
let $h\in V_1$ be as above such that $\sigma(h) = h$.
Let $(M, Y_M)$ be a $\sigma$-twisted $V$-module and
define $(M^{(h)}, Y_{M^{(h)}}(\cdot, z)) $ as follows:
\[
\begin{split}
& M^{(h)} =M \quad \text{ as a vector space;}\\
& Y_{M^{(h)}} (a, z) = Y_M(\Delta(h, z)a, z)\quad \text{ for any } a\in V.
\end{split}
\]
Then $(M^{(h)}, Y_{M^{(h)}}(\cdot, z))$ is a
$\sigma_h\sigma$-twisted $V$-module.
Furthermore, if $M$ is irreducible, then so is $M^{(h)}$.
\end{proposition}

Assume that $V$ is self-dual.
Then the $L(0)$-weights on $M^{(h)}$ are given by
\begin{equation}
L(0)+h_{(0)}+\frac{\langle h|h\rangle}{2}{\rm id}.\label{Eq:Lh}
\end{equation}
The following lemma is immediate from the equation above.

\begin{lemma}\label{Lem:wtmodule} Let $M$ be a $\sigma$-twisted $V$-module whose $L(0)$-weights 
belong to $\Z/3$. 
Let $h\in V_1$ such that $\sigma(h)=h$, $\langle h|h\rangle\in(2/3)\Z$, and $h_{(0)}$ is semisimple 
on $M$.
Assume that the spectra of $h_{(0)}$ on $M$ belong to $\Z/3$.
Then the $L(0)$-weights of the $\sigma_h\sigma$-twisted $V$-module $M^{(h)}$ also belong to $\Z/3$.
\end{lemma}

For a simple Lie algebra $\mathfrak{a}$ and a positive integer $k$, let $L_{\mathfrak{a}}(k,0)$ be 
the simple affine VOA associated with  $\mathfrak{a}$ at level $k$.
The following lemma was proved in \cite[Lemma 3.6]{LS3} for a simple Lie algebra.
One can easily generalize it to a semisimple Lie algebra.

\begin{lemma}\label{Lem:lowestwt}{\rm (cf.\ \cite[Lemma 3.6]{LS3})} Let $\mathfrak{g}$ be a 
semisimple  Lie algebra and let $\mathfrak{g}=\bigoplus_{i=1}^t\mathfrak{g}_i$ be the decomposition 
of $\mathfrak{g}$ into the direct sum of simple ideals $\mathfrak{g}_i$.
Let $k_i$ be a positive integer and let $\lambda_i$ be a dominant integral weight of 
$\mathfrak{g}_i$ (with respect to a fixed Cartan subalgebra).
Let $L_{\Fg_i}(k_i,\lambda_i)$ be the irreducible $L_{\Fg_i}(k_i,0)$-module with highest weight 
$\lambda_i$.
Let $h$ be an element in the Cartan subalgebra such that $$(h|\alpha)\ge-1$$ for any root $\alpha$ 
of $\mathfrak{g}$.
Then the lowest $L(0)$-weight of $\left(\bigotimes_{i=1}^t L_{\Fg_i}(k_i,\lambda_i)\right)^{(h)}$ is 
equal to \begin{equation}
\ell+\sum_{i=1}^t\min\{(h_i|\mu)\mid \mu\in\Pi(\lambda_i)\}+\frac{\langle 
h|h\rangle}{2},\label{Eq:twisttop}
\end{equation}
where $\ell$ is the lowest $L(0)$-weight of $\bigotimes_{i=1}^tL_{\mathfrak{g}_i}(k_i,\lambda_i)$ , $h=\sum_{i=1}^th_i$ $(h_i\in\g_i)$, 
and $\Pi(\lambda_i)$ is the set of all weights of the irreducible $\mathfrak{g}_i$-module with the 
highest weight $\lambda_i$.
\end{lemma}

\section{Conjugacy classes of the automorphism group of a lattice VOA}
In this section, we discuss the conjugacy classes of certain order $3$ automorphisms of Niemeier lattice VOAs.

\subsection{Conjugation under the inner automorphism group of a simple Lie algebra}

By the similar arguments as in \cite[Theorem 8.6]{Kac} (cf.\ \cite[Exercise 8.10]{Kac}), one can  
obtain the following lemmas:

\begin{lemma}\label{Lem:conE6} Let $\mathfrak{g}$ be a simple Lie algebra of type $E_6$. Then two 
order $3$ automorphisms of $\mathfrak{g}$ with the fixed point Lie algebra of type $A_2^3$ are  
conjugate under inner automorphisms of $\g$
\end{lemma}

\begin{lemma}\label{Lem:conD4} Let $\mathfrak{g}$ be a simple Lie algebra of type $D_4$.
\begin{enumerate}[{\rm (1)}]
\item Let $g$ and $g'$ be order $3$ automorphisms of $\mathfrak{g}$ with the fixed point Lie algebra of type $A_2$.
If $g^{-1}g'$ is inner, then $g$ and $g'$ are conjugate under inner automorphisms of $\g$.
\item Two order $3$ automorphisms of $\mathfrak{g}$ with the fixed point Lie algebra of type $A_1^3U(1)$ are conjugate under inner automorphisms of $\g$.
\end{enumerate}
\end{lemma}

The following lemma will be used later.

\begin{lemma}\label{L:perm3} Let $\g_1,\g_2,\g_3$ be isomorphic reductive Lie algebras and set $\g=\g_1\oplus\g_2\oplus\g_3$.
Let $g$ be an automorphism of $\g$ such that $g(\mathfrak{g}_1)=\mathfrak{g}_{2}$, $g(\mathfrak{g}_2)=\mathfrak{g}_{3}$ and $g(\mathfrak{g}_3)=\mathfrak{g}_{1}$.
Let $f$ be an inner automorphism of $\mathfrak{g}$.
If the order of $fg$ is $3$, then there exists an inner automorphism $x$ of $\g$ such that $xgx^{-1}=fg$.
\end{lemma}
\begin{proof} For $1\le i\le 3$, let $f_i$ be the inner automorphism of $\mathfrak{g}_i$ such that $f=f_1f_2f_3$.
It follows from $(fg)^3=1$ and $g^3=1$ that  
$f_1f_3^{g^2}f_2^g=f_2f_1^{g^2}f_3^g=f_3f_2^{g^2}f_1^g=1,$   
where $a^b=b^{-1}ab$.
Set $x=f_1f_2f_1^{g^2}$.
Then one can directly check that $xgx^{-1}=fg$.
\end{proof}

\subsection{Conjugation under the inner automorphism group of a VOA}
In this subsection, we prove some general results on conjugation under the inner automorphism group of a VOA.

\begin{lemma}\label{Lem:con4} Let $V$ be a VOA of CFT-type.
Assume that $V_1$ is reductive and let $\mathfrak{h}$ be a Cartan subalgebra of $V_1$.
Let $g$ be an order $3$ automorphism of $V$ such that
$g(\mathfrak{h})=\mathfrak{h}$.
Let $u\in\mathfrak{h}$ such that $\sum_{i=0}^{2} g^i(u)=0$.
Then $\sigma_ug$ is conjugate to $g$ under the inner automorphism group.
\end{lemma}
\begin{proof} Since $g\sigma_u=\sigma_{g(u)}g$ and $\sigma_u^{-1}=\sigma_{-u}$, we have  
\begin{align*}
&\sigma_{\frac13(g^2(u)-u)}\sigma_ug(\sigma_{\frac13(g^2(u)-u)})^{-1}
=\sigma_{\frac13(g^2(u)-u)+u+\frac13(-u+g(u))}g=\sigma_{\frac13(u+g(u)+g^2(u))}g=g
\end{align*}
as desired.
\end{proof}

\begin{proposition}\label{Prop:con3} Let $V$ be a self-dual VOA of CFT-type and $g$ an order $3$ automorphism of $V$.
Assume that the Lie algebra $V_1$ is reductive and let $V_1=\bigoplus_{i=1}^s\mathfrak{g}_i\oplus\mathfrak{a}$ be the decomposition into the direct sum of simple ideals $\mathfrak{g}_i$ with level $k_i$ and the radical $\mathfrak{a}$.
Then $g(\mathfrak{a})=\mathfrak{a}$ and $g$ acts on $\{\mathfrak{g}_i\mid 1\le i\le s\}$ as a permutation.
Without loss of generality, we may assume that there exists a non-negative integer $t$ such that $3t\le s$  and 
\[
g(\mathfrak{g}_i)=\begin{cases}\mathfrak{g}_{i+t}&{\rm if}\quad 1\le i\le 2t,\\ \mathfrak{g}_{i-2t}&{\rm if}\quad 2t+1\le i\le 3t,\\ \mathfrak{g}_i &{\rm if}\quad 3t+1\le i\le s.\end{cases}
\]
Then $$V^g_1=\bigoplus_{i=1}^t(\mathfrak{g}_i\oplus\mathfrak{g}_{i+t}\oplus\mathfrak{g}_{i+2t})^g\oplus\bigoplus_{i=3t+1}^s\mathfrak{g}_i^g\oplus\mathfrak{a}^g$$ is a direct sum of ideals.
Moreover
\begin{enumerate}[{\rm (1)}]
\item For $1\le i\le t$, we must have an equality of the levels $k_i=k_{i+t}=k_{i+2t}$. In addition, the Lie subalgebra $(\mathfrak{g}_i\oplus\mathfrak{g}_{i+t}\oplus\mathfrak{g}_{i+2t})^g$ is a simple ideal of $V^g_1$ isomorphic to $\mathfrak{g}_i$ and its level is $3k_i$.
\item Let $3t+1\le i\le s$.
Assume that the type of $\mathfrak{g}_i$ is $A_n$, $D_n$ or $E_n$.
If the restriction $g_{|\mathfrak{g}_i}$ of $g$ to $\mathfrak{g}_i$ is inner, then the Lie ranks of $\mathfrak{g}_i$ and $\mathfrak{g}_i^g$ are the same, and the level of any simple ideal of $\mathfrak{g}_i^g$ is $k_i$.
\item Let $3t+1\le i\le s$.
If the restriction $g_{|\mathfrak{g}_i}$ of $g$ to $\mathfrak{g}_i$ is not inner, then $\mathfrak{g}_i$ has the type $D_{4,k_i}$, and  $\mathfrak{g}_i^g$ has the type $A_{2,3k_i}$ or $G_{2,k_i}$.
\end{enumerate}
\end{proposition}

\begin{proof} 
Since the decomposition of $\mathfrak{g}$ into a direct sum of simple ideals and  the radical is 
unique, $g$ preserves the radical and permutes the simple ideals.

(1) Clearly,  the levels of $\mathfrak{g}_i$, $g(\mathfrak{g}_i)$ and $g^2(\mathfrak{g}_i)$ 
are the same. For $1\le i\le t$, set  
$\mathfrak{b}=\mathfrak{g}_i\oplus\mathfrak{g}_{i+k}\oplus\mathfrak{g}_{i+2k}$. Then the 
map $\nu:\mathfrak{g}_i\to\mathfrak{b}^g$, $x\mapsto x+g(x)+g^2(x)$ is an isomorphism of Lie 
algebras,  and $ \mathfrak{b}^g( \cong \mathfrak{g}_i)$ is a simple ideal of $\mathfrak{g}^g$. The 
level of $\mathfrak{b}^g$ can be computed easily by a direct calculation.

(2) Since $g_{|\mathfrak{g}_i}$ is inner, 
there is a common Cartan subalgebra for $\mathfrak{g}_i$ and $\mathfrak{g}_i^g$, and 
they have the same Lie ranks. 
Let $\mathfrak{c}$ be a simple ideal of $\mathfrak{g}_i^g$.
Since $\mathfrak{g}_i$ and $\mathfrak{g}_i^g$ share a Cartan subalgebra and the type of $\mathfrak{g}_i$ is $A_n$, $D_n$ or $E_n$, any (long) root of $\mathfrak{c}$ is also a (long) root of $\mathfrak{g}_i$.
By Proposition \ref{Prop:level}, the levels of $\mathfrak{c}$ and $\mathfrak{g}_i$ are the same.

(3) By the assumption, $\mathfrak{g}_i$ has a diagram automorphism of order $3$, and  hence the type 
of $\mathfrak{g}_i$ must be $D_4$. Then the Lie algebra of $\mathfrak{g}_i^g$ is either $A_2$ or 
$G_2$. Again, the level can be computed easily by a direct calculation.  
\end{proof}

\subsection{Certain order $3$ automorphisms of Niemeier lattice VOAs}\label{Sec:AutVL}

In this subsection, we characterize certain conjugacy classes of the lattice VOAs associated with the Niemeier lattices $Ni(E_6^4)$ and $Ni(D_4^6)$.

Let $L$ be an even lattice with positive-definite symmetric bilinear form $(\cdot|\cdot)$.
Let $L^*=\{v\in \Q\otimes_\Z L\mid (v|L)\subset\Z\}$ be the dual lattice of $L$. 
Let $Q$ be the root 
sublattice of $L$, the sublattice generated by norm $2$ vectors of $L$.
For the explicit construction of lattice VOAs, see \cite{FLM}.

\begin{lemma}\label{Lem:con1} 
Assume that the ranks of $L$ and $Q$ are the same.
Then $$\{g\in\Aut (V_L)\mid g=1\ {\rm on}\ (V_L)_1\}=\{\sigma_u\mid u\in Q^*/L^*\}.$$
\end{lemma}
\begin{proof} It follows from the fact that $V_L$ is a $(L/Q)$-graded simple 
current extension of 
$V_Q$ if $L$ and $Q$ has the same rank.
\end{proof}

Let $\hat{L}=\{\pm e^\alpha\mid \alpha\in L\}$ be the central extension of $L$ by $\langle-1\rangle$ with the commutator relation $e^\beta e^\alpha=(-1)^{( \alpha| \beta)} e^\alpha e^\beta$.
Let $\Aut (\hat{L})$ be the set of all group automorphisms of $\hat L$ and $O(L)$ the isometry 
group of $L$. For $\varphi\in \Aut (\hat{L})$, we define the group automorphism $\bar{\varphi}$ of 
$L$ by $\varphi(e^\alpha)\in\{\pm e^{\bar\varphi(\alpha)}\}$, $\alpha\in L$.
Set $O(\hat{L})=\{\varphi\in\Aut(\hat L)\mid \bar\varphi\in O(L)\}.$  
It was proved in \cite[Proposition 5.4.1]{FLM} that there exists a short exact sequence 
\begin{equation*}
1\to {\rm Hom}(L,\Z_2)\to O(\hat{L})\ \bar{\to}\ O(L)\to 1.\label{Eq:OL}
\end{equation*}

We call $\varphi\in O(\hat{L})$ a \emph{lift} of $g$ if $\bar\varphi=g$.
A lift $\varphi\in O(\hat{L})$ of $g$ is called \emph{standard} if  
$\varphi(e^\alpha)=e^\alpha$ for $\alpha\in L^g=\{v\in L\mid g(v)=v\}$. 
Note that a standard lift of $g$ always exists (\cite{Le}) and its order is $|g|$ or $2|g|$ (see 
\cite{EMS,MoPhD} for detail).

Let $V_L$ be the lattice VOA associated with $L$.
Then $O(\hat{L})$ is a subgroup of $\Aut (V_L)$ (\cite{FLM}).
Let $K$ be the normal subgroup of $\Aut (V_L)$ generated by inner automorphisms.
By \cite{DN}, $\Aut (V)=K O(\hat{L})$.
Set $M=K\cap O(\hat{L})$.
Since $M$ contains ${\rm Hom}(L,\Z_2)$, we obtain $\Aut (V_L)/K\cong O(\hat{L})/M\cong O(L)/\overline{M}$.

Let $N$ be a Niemeier lattice, an even unimodular lattice of rank $24$, and assume that $N$ is not the Leech lattice.
Let $Q$ be the root sublattice of $N$ and let $W\subset O(N)$ be the Weyl group of $Q$.
It follows from \cite{CS} that there exists a subgroup $H$ of $O(N)$ such that $O(N)=W{:}H$, and  
$H$ is isomorphic to the automorphism group of the glue code $N/Q$.

\begin{lemma}\label{Lem:M} Let $N$ be a Niemeier lattice and let $Q$ be the root sublattice of $N$.
Let $W$ be the Weyl group of $Q$.
Then $\overline{M}=W$, or  equivalently, $\Aut (V_N)/K\cong O(N)/W\cong H$.
\end{lemma}
\begin{proof} Since a lift of a reflection of $Q$ belongs to $K$ (cf.\ \cite{Hu}), we have 
$\overline{M}\supset W$.
Since $M$ preserves every irreducible $V_{Q}$-module, $\overline{M}$ acts trivially on $Q^*/Q$.
Hence $\overline{M}\subset W$.
\end{proof}

\subsubsection{Niemeier lattice with root lattice $E_6^4$}\label{S:E64}

Let $N=Ni(E_6^4)$ be a Niemeier lattice with root sublattice $Q\cong E_6^4$.
We now fix the glue code $\langle [1012], [1120],[1201]\rangle$ of $N$ as in \cite{CS} (see also \cite[Appendix A.1]{SS}).
Note that $E_6^*/E_6\cong\Z_3$ and that $N/Q\cong\Z_3^{2}$.

Let $\varphi$ be an order $3$ fixed-point-free isometry of the root lattice $E_6$.
Let $\bar{\sigma}_6$ be the order $3$ isometry of $N$ defined by
\begin{equation}
\bar{\sigma}_6:(\gamma_1,\gamma_2,\gamma_3,\gamma_4)\mapsto(\varphi(\gamma_1),\gamma_4,\gamma_2,\gamma_3).\label{Eq:sigma6}
\end{equation}
Let $\sigma_6\in O(\hat{N})$ be a standard lift of $\bar{\sigma}_6$ (cf. \cite[Appendix A]{SS} and \cite{Mi3}).
Note that the order of $\sigma_6$ is also $3$ and that the fixed-point weight $1$ subspace $(V_N^{\sigma_6})_1$ of $\sigma_6$ in $(V_N)_1$ has the Lie algebra structure $E_{6,3}A_{2,1}^3$ (see \cite{Mi3}).

\begin{proposition}\label{Prop:ConjE6} Let $g$ be an order $3$ automorphism of $V_N$.
Assume the following:
\begin{enumerate}[{\rm (1)}]
\item the Lie algebra structure of the fixed-point weight $1$ subspace $(V_N^g)_1$ is $E_{6,3}A_{2,1}^3$;
\item the weight one space of the irreducible $g$-twisted $V_N$-module is non-zero.
\end{enumerate}
Then $g$ is conjugate to $\sigma_6$ by an element of $\Aut (V_N)$.
\end{proposition}

\begin{proof} For $N=Ni(E_6^4)$, 
$H$ has the shape $2.S_4$ and it has exactly one conjugacy class of order $3$.
By Lemma \ref{Lem:M}, there exists a group homomorphism $p:\Aut (V_N) \to H\cong 2.S_4$ with the kernel $K$.
By \eqref{Eq:sigma6}, the assumption (1) and Proposition \ref{Prop:con3}, $p(g)$ is conjugate to 
$p(\sigma_6)$  by an element of $H$. Hence, we may assume that 
\begin{equation}
g=x\sigma_6,\qquad x\in K. \label{Eq:xs6}
\end{equation} 

Let $(V_N)_1\cong\bigoplus_{i=1}^4\mathfrak{g}_i$, where $\mathfrak{g}_i$ is a simple Lie algebra of type $E_6$.  
Then  every $\mathfrak{g}_i$ is preserved by elements of $K$ and we have 
$g(\mathfrak{g}_1)=\mathfrak{g}_1$, $g(\mathfrak{g}_2)=\mathfrak{g}_3$, 
$g(\mathfrak{g}_3)=\mathfrak{g}_4$ and $g(\mathfrak{g}_4)=\mathfrak{g}_2$ by \eqref{Eq:sigma6} and 
\eqref{Eq:xs6}.
Let $K_i$ be the subgroup of $\Aut (V_N)$ generated by inner automorphisms of $\mathfrak{g}_i$.
By Lemma \ref{L:perm3}, we may assume that $x\sigma_6=\sigma_6$ on $\bigoplus_{i=2}^4\g_i$ up to conjugation by an element of $K_2K_3K_4$.
Since the fixed point Lie algebras for $\sigma_6$ and $x\sigma_6$ 
in $\mathfrak{g}_1$ are both $A_2^3$, by Lemma \ref{Lem:conE6}, we may also assume that 
$x\sigma_6=\sigma_6$ on $\mathfrak{g}_1$ up to conjugation by an element of $K_1$.
Thus,  $x=\sigma_u$ for some element $u\in Q^*/N(\cong\Z_3^2)$ by Lemma \ref{Lem:con1}. 
Let $P:\h\to \{a\in\h\mid \sigma_6(a)=a\}$ be the orthogonal projection  and set $u'=P(u)$.
Then $\sum_{i=0}^2g^i(u-u')=0$, and by Lemma \ref{Lem:con4}, $x\sigma_6$ is conjugate to $\sigma_{u'}\sigma_6$ by an element of $K$.

Recall that $Q^*/N$ is generated by $[1000]$ and $[0100]$ and that $P([1000])=0$ and $P([0100])=(1/3)[0111]$.
Hence we may assume $u'=0$ or $u'=\pm(1/3)[0111]$.

We will show that $u'=0$. Suppose, for a contradiction, that $u'=\pm(1/3)[0111]$.
Then $\langle u'|u'\rangle=(u'|u')=4/9$.
Note that $\sigma_6(u')=u'$. 
Let $V_N[\sigma_6]$ be the irreducible $\sigma_6$-twisted $V_N$-module.
It follows from $(V_N[\sigma_6])_1\neq0$ (cf.\ \cite{Mi3,SS}) that $L(0)$-weights of $V_N[\sigma_6]$ belong to $(1/3)\Z$.
Since $(P(N)|u')\subset(1/3)\Z$, the spectrum of $u'_{(0)}$ on $V_N[\sigma_6]$ belongs to $(1/3)\Z$ (cf.\ \cite[(2.2.8)]{SS}).
Hence by \eqref{Eq:Lh}, the $L(0)$-weights of irreducible $\sigma_{u'}\sigma_6$-twisted $V_N$-module belong to $2/9+(1/3)\Z$, which contradicts the assumption (2).
Thus $u'=0$ and we obtain this proposition.
\end{proof}

\subsubsection{Niemeier lattice with root lattice $D_4^6$}\label{S:D46}

Let $N=Ni(D_4^6)$ be a Niemeier lattice with root sublattice $Q\cong D_4^6$.
Consider the glue code of $N/Q$ generated by 
$$ [111111],\ [222222],\ [002332],\ [023320],\ [032023],\ [020233]$$ as in \cite{CS} (see also \cite[Section 4.2]{SS}).
Note that $D_4^*/D_4\cong\Z_2\times\Z_2$ and $N/Q\cong\Z_2^{6}$.
Let $\bar{\sigma}_2$ and $\bar{\sigma}_4$ be order $3$ isometries of $N$ defined by
\begin{align}
\bar{\sigma}_2:(\gamma_1,\gamma_2,\dots,\gamma_6)&\mapsto(\varphi(\gamma_1),\varphi(\gamma_2),\dots,\varphi(\gamma_6)),\label{Eq:sigma2}\\
\bar{\sigma}_4:(\gamma_1,\gamma_2,\dots,\gamma_6)&\mapsto(\psi(\gamma_1),\varphi(\gamma_2),\varphi^{-1}(\gamma_3),\gamma_6,\varphi^{-1}(\gamma_4),\psi(\gamma_5))\label{Eq:sigma4},
\end{align}
where $\varphi$ is an order $3$ fixed-point-free isometry of $D_4$ (not in the Weyl group) and $\psi$ is an order $3$ isometry of $D_4$ in the Weyl group.
Let $\sigma_2,\sigma_4\in O(\hat{N})$ be standard lifts of isometries $\bar{\sigma}_2,\bar{\sigma}_4$ (cf.\ \cite[Section 4.2]{SS} and \cite[Section 6.2]{SS}), respectively.
Note that $\sigma_2$ and $\sigma_4$ have also order $3$ and that the fixed-point subspaces $(V_N^{\sigma_2})_1$ and $(V_N^{\sigma_4})_1$ of $(V_N)_1$ have the Lie algebra structures $A_{2,3}^6$ and $A_{2,3}^2D_{4,3}A_{1,1}^3U(1)$, respectively.

\begin{proposition}\label{Prop:ConjD4} Let $g$ be an order $3$ automorphism of $V_N$.
\begin{enumerate}[{\rm (1)}]
\item If the Lie algebra structure of $(V_N^g)_1$ is $A_{2,3}^6$, then $g$ is conjugate to $\sigma_2$.
\item If the Lie algebra structure of $(V_N^g)_1$ is $A_{2,3}^2D_{4,3}A_{1,1}^3U(1)$, then $g$ is conjugate to $\sigma_4$.
\end{enumerate}
\end{proposition}

\begin{proof} 
The proof is similar to Proposition \ref{Prop:ConjE6}. First we note that $H$ has the shape 
$3.S_6$ and it has exactly three conjugacy classes 
of order $3$, which are represented by an element in the center and elements with cycle shapes 
$3^1$ and $3^2$ (see \cite[p4]{Atlas}).
Let   $p:\Aut (V_N) \to H\cong 3.S_6$ be the natural group homomorphism with the  kernel $K$ 
(see Lemma \ref{Lem:M}).  
Set $(V_N)_1\cong\bigoplus_{i=1}^6\mathfrak{g}_i$, where $\mathfrak{g}_i$ is a simple Lie algebra of type $D_4$. 
Let $K_i$ be the subgroup of $\Aut (V_N)$ generated by inner automorphisms of $\mathfrak{g}_i$.

(1) By Proposition \ref{Prop:con3}, we have 
$g(\mathfrak{g}_i)=\mathfrak{g}_i$ for $1\le i\le 6$. Since $\sigma_2$ also stabilizes all 
$\g_i$, $p(g)$ is conjugate in $H$ to 
$p(\sigma_2)$, an element in the center of $H$. Then by Lemma \ref{Lem:conD4} (1) and Lemma 
\ref{Lem:con1}, we may assume, up to conjugation, $g = \sigma_u \sigma_2$ for some $u\in Q^*/N$.  
Since $\sigma_2$ is fixed-point-free on the canonical Cartan subalgebra of $(V_N)_1$, we have $\sum_{i=0}^2 g^i(u)=0$.
By applying Lemma \ref{Lem:con4} to this case, $g = \sigma_u \sigma_2$ is conjugate to $\sigma_2$, 
which proves (1).

(2) By \eqref{Eq:sigma4} and Proposition \ref{Prop:con3}, $p(g)$ is conjugate in $H$ to 
$p(\sigma_4)$, an element with cycle shape $3^1$ in $H$. Thus, we may assume that 
$
g=x\sigma_4, \ x\in K.
$
In this case, we have $g(\mathfrak{g}_i)=\mathfrak{g}_i$ 
for $1\le i\le 3$, $g(\mathfrak{g}_4)=\mathfrak{g}_5$, $g(\mathfrak{g}_5)=\mathfrak{g}_6$ and 
$g(\mathfrak{g}_6)=\mathfrak{g}_4$.
By Lemma \ref{L:perm3}, we may assume that $x\sigma_4=\sigma_4$ on $\bigoplus_{i=4}^6\g_i$ up to conjugation by an element of $K_4K_5K_6$.
By Lemma \ref{Lem:conD4} (1) and (2), we  may also assume that $x\sigma_4=\sigma_4$ on 
$\bigoplus_{i=1}^3\mathfrak{g}_i$ up to conjugation by an element of $K_1K_2K_3$.
Hence $x\sigma_4=\sigma_4$ on $(V_N)_1$. 
By Lemma \ref{Lem:con1}, $x=\sigma_u$ for some $u\in Q^*/N$.
Since $Q^*/N\cong\Z_2^6$, the order of $\sigma_u$ is $1$ or $2$.
Now suppose $\sigma_u$ has order $2$. 
Since $\sigma_u\sigma_4$ has order $3$, so does $(\sigma_u\sigma_4)^{-1}=\sigma_4^{-1}\sigma_u$.
Hence  $$
(\sigma_4\sigma_u\sigma_4^{-1})^{-1}(\sigma_u\sigma_4)(\sigma_4\sigma_u\sigma_4^{-1}
)=\sigma_4\sigma_u(\sigma_4^{-1}\sigma_u)^3\sigma_u=\sigma_4$$
and we have the desired result.
\end{proof}

\section{$\Z_n$-orbifold construction and reverse orbifold construction}

\subsection{$\Z_n$-orbifold construction}\label{S:Orb}

In this subsection, we will review the $\Z_n$-orbifold construction associated with a holomorphic 
VOA and an automorphism of arbitrary finite order from \cite{EMS,MoPhD}.

Let $V$ be a strongly regular holomorphic VOA.
Let $g$ be an order $n$ automorphism of $V$.
For $0\le i\le n-1$, we denote  the unique irreducible $g^i$-twisted $V$-module by $V[g^i]$ (cf. 
\cite[Theorem 1.2]{DLM2}). 
Note that $V[g^0]=V$.
Moreover, there exists an action $\phi_j: \langle g\rangle \to \Aut_\C (V[g^j])$ such that for all 
$v\in V$ and $i\in\Z,$  
$$\phi_j(g^i)Y_{V[g^j]}(v,z)\phi_j(g^i)^{-1}=Y_{V[g^j]}(g^iv,z).$$ Note that such an action is 
unique up to a multiplication of an $n$-th root of unity.
Set $\phi_0(g)=g\in\Aut (V)$.  For $0\leq j, k\leq n-1$, denote 
\[
W^{(j,k)} = \{ w\in V[g^j] \mid \phi_j(g) w = e^{(2\pi k \sqrt{-1})/n}w\}.
\]
Let $V^g$ be the fixed-point subspace of $g$, which is a full subVOA of $V$. Note that 
$W^{(0,0)}=V^g$ and all $W^{(j,k)}$'s are irreducible $V^g$-modules (cf. \cite[Theorem 2]{MT}). It 
was also shown recently in \cite{CM,Mi} that $V^g$ is strongly regular.
Moreover, any irreducible $V^g$-module is a submodule of $V$ or $V[g^i]$ for some $i$, and  there 
exist exactly $n^2$ non-isomorphic irreducible $V^g$-modules (see \cite{DRX}), which can be  
represented by $\{W^{(j,k)}\mid 0\le j,k\le n-1\}$.
By calculating the $S$-matrix of $V^g$, it was proved in \cite{EMS,MoPhD} that all irreducible 
$V^g$-modules $W^{(j,k)}$ are simple current modules. It implies that the set of isomorphism 
classes 
of irreducible $V^g$-modules, denoted by $R(V^g)$, forms an abelian group of order $n^2$ under the 
fusion product. 
We often identify an element in $R(V^g)$ with its representative irreducible $V^g$-module. Then 
$R(V^g)=\{W^{(j,k)}\mid 0\le j,k\le n-1\}$.

In addition, we assume the following:
\begin{enumerate}[{\rm (I)}]
\item For $1\le i\le n-1$, the lowest $L(0)$-weight of $V[g^i]$ belongs to $(1/n)\Z_{>0}$.
\end{enumerate}
Under the above assumption,  it is proved in \cite{EMS,MoPhD} that the abelian group $R(V^g)$ is 
isomorphic to $\Z_n\times\Z_n$.
Moreover, one can choose the $\phi_i$'s such that
\begin{itemize}
\item $W^{(i,j)}\fusion_{V^g} W^{(k, \ell)} \cong  W^{(i+k, j+\ell)}$, where $\fusion_{V^g}$ is the 
fusion product of $V^g$-modules;
\item the lowest $L(0)$-weight of $W^{(i,j)}$ belongs to $ji/n+\Z$.
\end{itemize}

\begin{theorem}[\cite{EMS,MoPhD}]\label{Thm:EMS} The $V^g$-module 
$$\widetilde{V}_g=\bigoplus_{i=0}^{n-1}W^{(i,0)}$$ has a strongly regular holomorphic VOA structure 
as a $\Z_n$-graded simple current extension of $V^g$.
\end{theorem}
The construction of $\widetilde{V}_g$ is often called the \emph{$\Z_n$-orbifold construction} 
associated with $V$ and $g$.
Note that $W^{(1,0)}$ is the unique irreducible $V^g$-submodule of $V[g]$ with integral weights and 
$\{ W^{(i,0)}\mid 0\leq i\leq n-1\}$ is the subgroup of $R(V^g)$ generated by $W^{(1,0)}$.
Hence $\widetilde{V}_g$ is uniquely determined by $V$ and $g$, up to isomorphism.

\begin{remark}\label{R:Uni}
\begin{enumerate}[(1)]
\item When $n$ is prime,  $W^{(i,0)}$ is just the subspace of $V[g^i]$ with integral weights for 
$1\le i\le n-1$.
\item Let $g'$ be an automorphism of $V$ which is conjugate to $g$.
Then $g'$ also satisfies Condition (I), and $\widetilde{V}_{g'}$ is isomorphic to $\widetilde{V}_g$ 
as a VOA.
\end{enumerate}
\end{remark}

The following dimension formula is described in \cite[(73)]{Mo}. A proof can be found in 
\cite{MoPhD} (see \cite{EMS2} for more general formulas).

\begin{theorem} [\cite{MoPhD,EMS2}] \label{Thm:Dimformula}
Let $V$ be a strongly regular holomorphic VOA of central charge $24$.
Let $g$ be an automorphism of $V$ of order $3$ such that Condition (I) holds. Then 
\begin{align*}
\dim V_1+\dim (\tilde{V}_g)_1=&4\dim (V^g)_1
-36(\dim V[g]_{1/3}+\dim V[g^2]_{1/3})\\
&-12(\dim V[g]_{2/3}+\dim V[g^2]_{2/3})
+24.
\end{align*}
\end{theorem}

\subsection{Reverse orbifold construction and uniqueness of a holomorphic VOA}\label{sec:4.2}

Let $V$ be a strongly regular holomorphic VOA and let $g$ be an order $n$ automorphism of $V$ 
satisfying Condition (I).
Let $W=\widetilde{V}_g$ be the resulting holomorphic VOA by applying $\Z_n$-orbifold construction to 
$V$ and $g$.
Then the $\Z_n$-grading of $W$ defines an automorphism $f$ of an order $n$ on $W$. In this case, 
$W^f=V^g$ and, for $1\le i\le n-1$, the irreducible $f^i$-twisted $W$-module is a direct sum of 
irreducible $V^g$-modules.
Hence $f$ also satisfies Condition (I).
By the uniqueness of the VOA  obtained by the orbifold construction, we obtain the 
following:
\begin{corollary}[cf.\ \cite{EMS,MoPhD}]\label{C:RevO}
The VOA $\widetilde{W}_f$ is isomorphic to $V$.
\end{corollary}
Since the VOA $\widetilde{W}_f\cong V$ is essentially obtained by reversing the original orbifold 
construction, we call this procedure the \emph{reverse orbifold construction}, which is called the 
inverse orbifold in \cite{EMS,MoPhD}.
Based on the above corollary, one can prove the uniqueness of some holomorphic VOAs as follows:

Let $\g$ be a Lie algebra and $\mathfrak{p}$ a subalgebra of $\g$. 
Let $n\in \Z_{>0}$ and let $W$ be a strongly regular holomorphic VOA of central charge $c$.
Assume that for any strongly regular holomorphic VOA $V$ of central charge $c$ whose weight one Lie 
algebra is $\g$, there exists an order $n$ automorphism $\sigma$ of $V$ such that the following 
conditions hold:
\begin{enumerate}[{\rm (a)}]
\item $\g^{\sigma}\cong\mathfrak{p}$;
\item $\sigma$ satisfies Condition (I) and $\widetilde{V}_{\sigma}$ is isomorphic to $W$.
\end{enumerate}
In addition, we assume that any automorphism  $\varphi\in\Aut (W)$ of order $n$ satisfying (I) and 
the conditions (A) and (B) below belongs to a unique conjugacy class in $\Aut (W)$:
\begin{enumerate}[{\rm (A)}]
\item $(W^\varphi)_1$ is isomorphic to $\mathfrak{p}$;
\item $(\widetilde{W}_\varphi)_1$ is isomorphic to $\g$.
\end{enumerate}
Then any strongly regular holomorphic VOA of central charge $c$ with weight one Lie algebra $\g$ is 
isomorphic to $\widetilde{W}_\varphi$.
On the other hand, we can identify $\widetilde{V}_{\sigma}$ with $W$.
Now let $\tau$ be an order $n$ automorphism of $W$ associated with  the $\Z_n$-grading of $W$ as 
the extension of $V^{\sigma}$.
Then $(W^\tau)_1\cong V^{\sigma}_1\cong\mathfrak{p}$ by (a) and $\widetilde{W}_\tau\cong V$ by 
Corollary \ref{C:RevO}.
Hence $(\widetilde{W}_\tau)_1\cong V_1=\g$.
Since $\tau$ satisfies (A) and (B), it is conjugate to $\varphi$.
Thus $V\cong \widetilde{W}_\tau\cong \widetilde{W}_\varphi$ and the VOA structure of $V$ is 
uniquely determined. 

\begin{remark}\label{R:RO}
\begin{enumerate}[(1)]
\item In some cases, the condition (B) can be replaced by a weaker condition.
For example, we will consider the 
following condition (B') in the 
later sections:  
\begin{enumerate}
\item[(B')] $\dim W[\varphi]_1\neq0$.
\end{enumerate}
\item It turns out that we may choose $\sigma$ as an inner automorphism $\sigma_u$ for 
some $u$ in a Cartan subalgebra.
\end{enumerate}
\end{remark}

\section{Uniqueness of a holomorphic VOA with the Lie algebra $E_{6,3}G_{2,1}^3$}
In this section, we consider a  holomorphic VOA $V$ of central charge $24$ whose weight one Lie 
algebra has the type  $E_{6,3}G_{2,1}^3$. By applying the $\Z_3$-orbifold construction to $V$ 
and certain inner automorphism of order $3$, we show that one can obtain a holomorphic lattice VOA 
of central charge $24$ associated with the Niemeier lattice $Ni(E_{6}^4)$.
In addition,  we prove that the structure of a strongly regular holomorphic VOA of central charge $24$ is uniquely determined by its weight one Lie algebra if the Lie algebra has 
the type $E_{6,3}G_{2,1}^3$. 

\subsection{Simple affine VOA of type $G_{2}$ at level $1$}\label{Sec:G2}
Let $\alpha_1$ and $\alpha_2$ be simple roots of type $G_2$ such that $(\alpha_1|\alpha_1)=2/3$, $(\alpha_2|\alpha_2)=2$ and $(\alpha_1|\alpha_2)=-1$.
Let $\Lambda_1$ and $\Lambda_2$ be the fundamental weights with respect to $\alpha_1$ and $\alpha_2$, respectively.
Note that $2(\Lambda_i|\alpha_j)/(\alpha_j|\alpha_j)=\delta_{i,j}$ for $i,j\in\{1,2\}$.
Let $L_\Fg(1,0)$ be the simple affine VOA associated with the simple Lie algebra $\Fg$ of type $G_2$ at level $1$.
It is well-known (cf. \cite{FZ}) that there exist exactly two (non-isomorphic) irreducible $L_\Fg(1,0)$-modules $L_\Fg(1,0)$ and $L_\Fg(1,\Lambda_1)$.

For a dominant integral weight $\lambda$ of $\mathfrak{g}$, set $$n_{\Lambda_1}(\lambda)=\min\{(\Lambda_1|\mu)\mid \mu\in\Pi(\lambda)\}.$$
Here $\Pi(\lambda)$ denotes the set of all weights of the irreducible $\mathfrak{g}$-module with the  highest weight $\lambda$.
One can easily verify the following lemma.

\begin{lemma}\label{Lem:G21}
The values of $n_{\Lambda_1}(\lambda)$ are given as follows:
\begin{center}
\begin{tabular}{|c|c|c|c|}
\hline
Weight $\lambda$ & Lowest $L(0)$-weight of $L_\mathfrak{g}(1,\lambda)$ & $(\Lambda_1|\lambda)$ & $n_{\Lambda_1}(\lambda)$ \\
\hline \hline 
$0$ & $0$& $0$& $0$ \\\hline 
$\Lambda_1$&$2/5$&$2/3$&$-2/3$\\
\hline
\end{tabular}
\end{center}
\end{lemma}

\subsection{Inner automorphism of a holomorphic VOA with Lie algebra $E_{6,3}G_{2,1}^3$}\label{S:E6-2}

Let $V$ be a strongly regular holomorphic VOA of central charge $24$ whose weight one Lie algebra has the type $E_{6,3}G_{2,1}^3$.
Let $V_1=\bigoplus_{i=1}^4\mathfrak{g}_i$ be the decomposition into the direct sum of simple ideals, where 
 the type of $\mathfrak{g}_1$ is $E_{6,3}$, and the types of $\mathfrak{g}_2$, $\mathfrak{g}_3$ and $\mathfrak{g}_4$ are $G_{2,1}$.
Let $\mathfrak{H}$ be a Cartan subalgebra of $V_1$.
Then $\mathfrak{H}\cap\mathfrak{g}_i$ is a Cartan subalgebra of $\mathfrak{g}_i$.
Let $U$ be the subVOA generated by $V_1$.
Set $k_1=3$ and $k_i=1$ for $i=2,3,4$.
By Proposition \ref{Prop:posl}, $U\cong \bigotimes_{i=1}^4L_{\mathfrak{g}_i}(k_i,0)$.

Let 
\begin{equation}
{h}=(0,\Lambda_1,\Lambda_1,\Lambda_1)\in\bigoplus_{i=1}^4(\mathfrak{H}\cap\mathfrak{g}_i).\label{Def:hE6}
\end{equation}
Note that 
\begin{equation}
\langle h|h\rangle=3\times(\Lambda_1|\Lambda_1)_{|\mathfrak{g}_2}=2.\label{Eq:h1}
\end{equation}
Since $(\Lambda_1|\alpha)\ge-1$ for any root of the simple Lie algebra of type $G_2$, we have 
\begin{equation}
(h|\alpha)\ge-1\label{Eq:ha}
\end{equation}
for all roots $\alpha$ of the semisimple Lie algebra $V_1$.

\begin{lemma}\label{Lem:order2} 
The order of $\sigma_h$ is $3$ on $V$, and the $L(0)$-weights of the irreducible $\sigma_h$-twisted (resp. $\sigma_h^{-1}$-twisted) $V$-module $V^{(h)}$ (resp. $V^{(-h)}$) belong to $\Z/3$.
\end{lemma}
\begin{proof} 
Since $(\Lambda_1|\lambda)\in\Z/3$ for any weight $\lambda$ of the simple Lie algebra of type $G_2$, we have $(\sigma_h)^3=1$.
In addition, $\sigma_h$ is an automorphism of order $3$ on $V_1$, so is it on $V$.
The latter assertion follows from Lemmas \ref{Lem:wtmodule} and \ref{Lem:G21} and \eqref{Eq:h1}.
\end{proof}

\begin{proposition}\label{Prop:low} The lowest $L(0)$-weight of the irreducible $\sigma_h$-twisted (resp. $\sigma_{h}^{-1}$-twisted) $V$-module $V^{(h)}$ (resp. $V^{(-h)}$) is $1$.
In particular, $(V^{(\pm h)})_{i}=0$ for $i\in\{1/3,2/3\}$.
\end{proposition}
\begin{proof} Let $M\cong \bigotimes_{i=1}^4L_{\mathfrak{g}_i}(k_i,\lambda_i)$ be an irreducible $U$-submodule of $V$.
Let $\ell$ and $\ell^{(h)}$ be the lowest $L(0)$-weights of $M$ and of $M^{(h)}$, respectively.
By \eqref{Eq:ha}, we can apply Lemma \ref{Lem:lowestwt} to our case.
By \eqref{Eq:twisttop}, \eqref{Def:hE6} and \eqref{Eq:h1}, we have
\begin{align}
\ell^{(h)}
=\ell+\sum_{i=2}^4n_{\Lambda_1}(\lambda_i)+1,\label{Eq:lh}
\end{align}
where $n_{\Lambda_1}(\cdot)$ is defined as in the previous subsection.
If $(\lambda_1,\lambda_2,\lambda_3,\lambda_4)=(0,0,0,0)$, then $\ell=0$ and $n_{\Lambda_1}(0)=0$, and hence $\ell^{(h)}=1$.
Assume $(\lambda_1,\lambda_2,\lambda_3,\lambda_4)\neq(0,0,0,0)$.
Then $\ell\ge2$ and  $n_{\Lambda_1}(\lambda_i)\ge n_{\Lambda_1}(\Lambda_1)=-2/3$ by Lemma \ref{Lem:G21}. 
Hence 
$$\ell^{(h)}\ge 2+3n_{\Lambda_1}(\Lambda_1)+1=1.$$

One can also prove $\ell^{(-h)}\ge1$ by the same argument.
\end{proof}

\subsection{Identification of the Lie algebra: Case $E_{6,1}^4$}\label{S:E6-3}
Thanks to Proposition \ref{Prop:low}, we can apply Theorem \ref{Thm:EMS} to $V$ and $\sigma_h$, and 
obtain a strongly regular holomorphic VOA $\tilde{V}=\tilde{V}_{\sigma_h}$ of central charge $24$.
By the definition of $h$, we obtain the following:

\begin{proposition}\label{Prop:fixedE63} 
The Lie algebra structure of $(V^{\sigma_h})_1$ is $E_{6,3}A_{2,1}^3$ and $\dim(V^{\sigma_h})_1=102$.
\end{proposition}

Next we will identify the Lie algebra structure of  $\tilde{V}_1$.

\begin{proposition}\label{Prop:E64} 
The Lie algebra structure of $\tilde{V}_1$ is $E_{6,1}^4$.
In particular, $\tilde{V}$ is isomorphic to the lattice VOA associated with the Niemeier lattice $Ni(E_6^4)$.
\end{proposition}
\begin{proof} 
By Proposition \ref{Prop:V1}, the Lie algebra $\tilde{V}_1$ is semisimple.
By Proposition \ref{Prop:fixedE63}, we have $\dim (V^{\sigma_h})_1=102$, and 
by Proposition \ref{Prop:low}, we have $(V^{(\pm h)})_{i}=0$ for $i=1/3,2/3$.
By Theorem \ref{Thm:Dimformula}, $$\dim\tilde{V}_1=4\times \dim (V^{\sigma_h})_1-\dim V_1+24=312;$$
hence we obtain the ratio $h^\vee/k=12$ by Proposition \ref{Prop:V1}.
Since the level $k$ of any simple ideal is integral, the dual Coxeter number $h^\vee$ is a multiple of $12$ , which shows that possible types of simple ideals of $\tilde{V}_1$ are $A_{11,1}$, $C_{11,1}$, $D_{7,1}$ and $E_{6,1}$.
It follows from $\dim \tilde{V}_1=312$ that the Lie algebra structure of $\tilde{V}_1$ is $A_{11,1}D_{7,1}E_{6,1}$ or $E_{6,1}^4$.
Since $\tilde{V}_1$ contains a Lie subalgebra of the type $E_{6,3}$, the Lie algebra structure of $\tilde{V}_1$ must be $E_{6,1}^4$.
By Proposition \ref{P:NVOA}, we have proved this proposition.
\end{proof}

\subsection{Main theorem for the Lie algebra $E_{6,3}G_{2,1}^3$}
Set $N=Ni(E_6^4)$ and let $\sigma_6$ be the order $3$ automorphism of $V_N$ defined in Section \ref{S:E64}.
Then $\sigma_6$ satisfies Condition (I) in Section \ref{S:Orb} and the holomorphic VOA $(\widetilde{V_{N}})_{\sigma_6}$ has the weight one Lie algebra of type $E_{6,3}G_{2,1}^3$ (\cite{Mi3,SS}).
Note that the Lie algebra $(V_N^{\sigma_6})_1$ has the type $E_{6,3}A_{2,1}^3$.
We now prove the main theorem in this section.
\begin{theorem} Let $V$  be a strongly regular holomorphic VOA of central charge $24$ whose weight one Lie algebra has the type $E_{6,3}G_{2,1}^3$.
Then $V$ is isomorphic to the holomorphic VOA $(\widetilde{V_{N}})_{\sigma_6}$.
In particular, the structure of a strongly regular holomorphic VOA of central charge $24$ is uniquely determined by its weight one Lie algebra if the Lie algebra has the type $E_{6,3}G_{2,1}^3$.
\end{theorem}

\begin{proof} It suffices to verify the hypotheses in Section \ref{sec:4.2} for 
$\g=E_{6,3}G_{2,1}^3$, $\mathfrak{p}=E_{6,3}A_{2,1}^3$, $n=3$ and $W=V_{Ni(E_6^4)}$.
Take $h$ as in \eqref{Def:hE6} and set $\sigma=\sigma_h$.
Then by Lemma \ref{Lem:order2}, the order of $\sigma$ is $3$.
The former hypotheses (a) and (b) hold by Propositions \ref{Prop:low}, \ref{Prop:fixedE63} and \ref{Prop:E64}.
Note that the eigenspaces of $\sigma$ on $\g$ with eigenvalues $e^{\pm 2\pi\sqrt{-1}/3}$ are non-zero.
The latter hypotheses about the uniqueness of the conjugacy class follows from Proposition \ref{Prop:ConjE6}.
We remark that the assumption (2) in Proposition \ref{Prop:ConjE6} (or (B') in Remark \ref{R:RO} 
(1)) is weaker than the assumption (B). 
\end{proof}

\section{Uniqueness of a holomorphic VOA with the Lie algebra $A_{2,3}^{6}$}
In this section, we consider a  holomorphic VOA $V$ of central charge $24$ with the weight one Lie algebra $A_{2,3}^6$. By applying the $\Z_3$-orbifold construction to $V$ 
and certain inner automorphism of order $3$, we show that one can obtain a holomorphic lattice VOA of central charge $24$ associated with the Niemeier lattice $Ni(D_{4}^6)$.
In addition, the structure of a strongly regular holomorphic VOA of central charge $24$ is uniquely determined by its weight one Lie algebra if the Lie algebra has the type $A_{2,3}^6$.

\subsection{Lowest weights of irreducible modules over a simple Lie algebra of type $A_n$}
Let $\Phi$ be a root system of type $A_n$ and let $\Delta=\{\alpha_1,\dots,\alpha_n\}$ be a set of simple roots in $\Phi$ such that $(\alpha_i|\alpha_j)=2\delta_{i,j}-\delta_{|i-j|,1}$.
Let $W$ be the Weyl group of $\Phi$.
Let $\tau$ be the Dynkin diagram automorphism of $\Phi$ defined by $\tau(\alpha_i)=\alpha_{n-i+1}$, $1\le i\le n$.
Let $\gamma$ denote the $-1$-isometry of $\Phi$, that is, $\gamma(\alpha_i)=-\alpha_i$, $1\le i\le n$.
It is well-known (cf.\ \cite[Section 12.2, Exercise 13.5]{Hu}) that $\tau$ and $\gamma$ do not belong to $W$ and that $\tau\gamma\in W$.
Clearly, the order of $\tau\gamma$ is $2$.

Let $\mathfrak{g}$ be a simple Lie algebra associated with $\Phi$ and let $\lambda$ be a dominant integral weight of $\Phi$ with respect to $\Delta$.
Let $\Pi(\lambda)$ be the set of all weights of the irreducible $\mathfrak{g}$-module with highest weight $\lambda$.
Let $\prec$ be the partial order on $\R\Phi(\cong\R^n)$ such that $x\prec y$ if $y-x$ is a non-negative $\R$-linear combination of simple roots.
Note that $\mu\prec\lambda$ for any weight $\mu$ in $\Pi(\lambda)$.
The following lemma is well-known (cf.\ \cite{Hu}).

\begin{lemma} The lowest weight in $\Pi(\lambda)$ is $\tau\gamma(\lambda)$, that is, $\tau\gamma(\lambda)\prec \mu$ for any weight $\mu$ in $\Pi(\lambda)$.
\end{lemma}

The following corollary is immediate from the lemma above.

\begin{corollary}\label{Cor:LW} Let $\Lambda$ be a $\R_{\ge0}$-linear combination of fundamental weights.
Then $$\min\{(\Lambda|\mu)\mid \mu\in\Pi(\lambda)\}=(\Lambda|\tau\gamma(\lambda)).$$
\end{corollary}

\subsection{Simple affine VOA of type $A_{2}$ at level $3$}
Let $\alpha_1,\alpha_2$ be simple roots of type $A_2$ such that $(\alpha_1|\alpha_1)=(\alpha_2|\alpha_2)=2$ and $(\alpha_1|\alpha_2)=-1$.
Let $\Lambda_1$ and $\Lambda_2$ be the fundamental weights with respect to $\alpha_1$ and $\alpha_2$, respectively.
Let $L_\Fg(3,0)$ be the simple affine VOA associated with the simple Lie algebra $\Fg$ of type $A_2$ at level $3$.
It is well-known (cf. \cite{FZ}) that there exist exactly $10$ (non-isomorphic) irreducible $L_\Fg(3,0)$-modules $L_\Fg(3,\lambda)$ with the highest weight $\lambda$, where $\lambda$ ranges over $\{0,\Lambda_i, 2\Lambda_i,3\Lambda_i,\Lambda_1+\Lambda_2,2\Lambda_1+\Lambda_2,\Lambda_1+2\Lambda_2\mid i=1,2\}.$

For a dominant integral weight $\lambda$ of $\mathfrak{g}$, set $$n_{\Lambda_1}(\lambda)=\min\{(\Lambda_1|\mu)\mid \mu\in\Pi(\lambda)\}.$$
By Corollary \ref{Cor:LW} and $\tau\gamma(\Lambda_1)=-\Lambda_2$, we have $n_{\Lambda_1}(\lambda)=(\Lambda_1|\tau\gamma(\lambda))=-(\Lambda_2|\lambda)$.

\begin{lemma}\label{Lem:A23}
The values of $n_{\Lambda_1}(\lambda)$ are given as follows:
\begin{longtable}{|c|c|c|c|}
\hline
Weight $\lambda$ & Lowest $L(0)$-weight of $L_\mathfrak{g}(3,\lambda)$ & $(\Lambda_1|\lambda)$ & $n_{\Lambda_1}(\lambda)$
\\
\hline \hline 
$0$ & $0$& $0$& $0$ \\ \hline
$\Lambda_1$ & $2/9$ & $2/3$ & $-1/3$ \\ \hline
$\Lambda_2$ & $2/9$ & $1/3$ & $-2/3$ \\ \hline
$2\Lambda_1$ & $5/9$ & $4/3$ & $-2/3$\\  \hline
$2\Lambda_2$ & $5/9$ & $2/3$ & $-4/3$\\  \hline
$\Lambda_1+\Lambda_2$ & $1/2$ & $1$ & $-1$ \\ \hline
$3\Lambda_1$ & $1$ & $2$ & $-1$ \\ \hline
$3\Lambda_2$ & $1$ & $1$ & $-2$ \\ \hline
$2\Lambda_1+\Lambda_2$ & $8/9$ & $5/3$ & $-4/3$ \\ \hline
$\Lambda_1+2\Lambda_2$ & $8/9$ & $4/3$ & $-5/3$ \\
\hline
\end{longtable}
\end{lemma}

\subsection{Inner automorphism of a holomorphic VOA with Lie algebra $A_{2,3}^6$}\label{S:A23-2}

Let $V$ be a strongly regular holomorphic VOA of central charge $24$ whose weight one Lie algebra has the type $A_{2,3}^6$.
Let $V_1=\bigoplus_{i=1}^6\mathfrak{g}_i$ be the decomposition into the direct sum of simple ideals, where 
 the type of $\mathfrak{g}_i$ is $A_{2,3}$ for $1\le i\le 6$.
Let $\mathfrak{H}$ be a Cartan subalgebra of $V_1$.
Then $\mathfrak{H}\cap\mathfrak{g}_i$ is a Cartan subalgebra of $\mathfrak{g}_i$.
Let $U$ be the subVOA generated by $V_1$.
By Proposition \ref{Prop:posl}, $U\cong \bigotimes_{i=1}^6 L_{\mathfrak{g}_i}(3,0)$.

Let 
\begin{equation}
h=(\Lambda_1,0,0,0,0,0)\in \bigoplus_{i=1}^6(\mathfrak{H}\cap\mathfrak{g}_i).\label{Def:hA23}
\end{equation}
Note that
\begin{equation}
\langle h|h\rangle=3\times(\Lambda_1|\Lambda_1)_{|\mathfrak{g}_1}=2\label{Eq:h2}
\end{equation}
and that for any root $\alpha$ of $V_1$, we have
\begin{equation}
(h|\alpha)\ge-1.\label{Eq:ha2}
\end{equation}

Now, let us recall the following theorems:

\begin{theorem}{\rm (\cite[Theorem 2]{KMi})}\label{Thm:KMi}
Let $P$ be a strongly regular holomorphic VOA and let $Q$ be a simple regular subVOA of $P$.
Assume that the commutant subalgebra $Q^c$ is simple, regular, and that $(Q^c)^c=Q$.
Then all irreducible $Q$-modules appear in $P$ as $Q$-submodules.
\end{theorem}

\begin{theorem}{\rm (\cite[Theorem 3.5]{HKL})}\label{Thm:HKL}
Let $W$ be a regular simple VOA of CFT-type.
Assume that for any irreducible $W$-module, its lowest $L(0)$-weight is nonnegative except for $W$ itself.
Let $Q$ be a $C_2$-cofinite, simple VOA of CFT-type containing $W$ as a full subVOA.
Then $Q$ is rational.
\end{theorem}

Using the theorems above, we prove the following lemma:

\begin{lemma}\label{Lem:or3A23} 
The order of $\sigma_h$ is $3$ on $V$, and the $L(0)$-weights of the irreducible $\sigma_h$-twisted (resp. $\sigma_h^{-1}$-twisted) $V$-module $V^{(h)}$ (resp. $V^{(-h)}$) belong to $\Z/3$.
\end{lemma}
\begin{proof}
Let $M\cong \bigotimes_{i=1}^6 L_{\mathfrak{g}_i}(3,\lambda_i)$ be an irreducible $U$-submodule of $V$.
By Lemma \ref{Lem:A23}, $(h|(\lambda_1,\dots,\lambda_6))\in\Z/3$.
Hence the order of $\sigma_h$ is $1$ or $3$ on $V$.

Let $Q=\langle\mathfrak{g}_1\rangle_{\rm VOA}$ be the subVOA of $V$ generated by $\mathfrak{g}_1$ and let $Q^c$ be the commutant subalgebra of $Q$ in $V$.
Since the lowest $L(0)$-weight of any irreducible $Q$-module is less than $2$ (see Lemma \ref{Lem:A23}), we have $(Q^c)^c=Q$.
Let $W=\langle \mathfrak{g}_i\mid 2\le i\le 6\rangle_{\rm VOA}$ be the the subVOA of $V$ generated by $\bigoplus_{i=2}^6\mathfrak{g}_i$.
Note that $W$ is regular, and the $L(0)$-lowest weight of any irreducible $W$-module is positive except for $W$ itself.
Since $Q^c$ is an extension of $W$, $Q^c$ is $C_2$-cofinite by \cite{ABD}.
In addition, $Q^c$ is simple (see \cite[Lemma 2.1]{ACKL}).
Then by Theorem \ref{Thm:HKL}, $Q^c$ is rational, and by \cite{ABD}, $Q^c$ is regular.
Applying Theorem \ref{Thm:KMi} to $Q$, we know that every irreducible $Q$-submodule appears in $V$.
Hence by Lemma \ref{Lem:A23}, there exists a $Q\otimes W$-submodule with highest weight $(\lambda_1,\dots,\lambda_6)$ such that $(h|(\lambda_1,\dots,\lambda_6))\notin \Z$, which shows that $\sigma_h\neq1$ on $V$.

The latter assertion follows from Lemmas \ref{Lem:wtmodule} and \ref{Lem:A23} and \eqref{Eq:h2}. 
\end{proof}

\begin{proposition}\label{Prop:lowA23} The lowest $L(0)$-weight of the irreducible $\sigma_h$-twisted (resp. $\sigma_{h}^{-1}$-twisted) $V$-module $V^{(h)}$ (resp. $V^{(-h)}$) is $1$.
In particular, $(V^{(\pm h)})_{i}=0$ for $i\in\{1/3,2/3\}$.
\end{proposition}
\begin{proof} Let $M\cong \bigotimes_{i=1}^6 L_{\mathfrak{g}_i}(3,\lambda_i)$ be an irreducible $U$-submodule of $V$.
Let $\ell$ and $\ell^{(h)}$ be the lowest $L(0)$-weights of $M$ and of $M^{(h)}$, respectively.
By \eqref{Eq:ha2}, we can apply Lemma \ref{Lem:lowestwt} to our case.
By \eqref{Eq:twisttop}, \eqref{Def:hA23} and \eqref{Eq:h2}, we have
\begin{align*}
\ell^{(h)}=\ell+n_{\Lambda_1}(\lambda_1)+1,
\end{align*}
where $n_{\Lambda_1}(\Lambda)$ is defined in the previous subsection.
If $\lambda_i=0$ for $1\le i\le 6$, then $\ell=0$ and $\ell^{(h)}=1$.
If $\lambda_i\neq0$ for some $i$, then $\ell\ge2$ and, by Lemma \ref{Lem:A23} $n_{\Lambda_1}(\lambda_1)\ge -2$.
Hence $\ell^{(h)}\ge1$.

One can also prove $\ell^{(-h)}\ge1$ by the same argument.
\end{proof}

\subsection{Identification of the Lie algebra: Case $D_{4,1}^6$}\label{S:A23-3}
Thanks to Proposition \ref{Prop:lowA23}, we can apply Theorem \ref{Thm:EMS}  to $V$ and 
$\sigma_h$, and obtain the strongly regular 
holomorphic VOA $\tilde{V}=\tilde{V}_{\sigma_h}$ of central charge $24$. 
Since $\sigma_h$ acts trivially on $V_1$, we have $(V^{\sigma_h})_1= V_1$.

\begin{lemma}\label{Lem:fixA23} 
The Lie algebra structure of $(V^{\sigma_h})_1$ is $A_{2,3}^6$ and $\dim(V^{\sigma_h})_1=48$.
\end{lemma}

Next we will identify the Lie algebra structure of  $\tilde{V}_1$.

\begin{proposition}\label{Prop:D46} 
The Lie algebra structure of $\tilde{V}_1$ is $D_{4,1}^6$.
In particular, $\tilde{V}$ is isomorphic to the lattice VOA associated with the Niemeier lattice $Ni(D_4^6)$.
\end{proposition}
\begin{proof} 
By Proposition \ref{Prop:V1}, the Lie algebra $\tilde{V}_1$ is semisimple.
By Proposition \ref{Lem:fixA23}, we have $\dim (V^{\sigma_h})_1=48$, and 
by Proposition \ref{Prop:lowA23}, we have $(V^{(\pm h)})_{i}=0$ for $i=1/3,2/3$.
By Theorem \ref{Thm:Dimformula}, $$\dim\tilde{V}_1=4\times \dim (V^{\sigma_h})_1-\dim V_1+24=168;$$
hence we obtain the ratio $h^\vee/k=6$ by Proposition \ref{Prop:V1}.
It also follows from Proposition \ref{Prop:V1} that $\tilde{V}_1$ is semisimple.
Since the level $k$ of any simple ideal of $\tilde{V}_1$ is integral, its dual Coxeter number $h^\vee$ is a multiple of $6$, which shows that the  possible types of simple ideals of $\tilde{V}_1$ are $A_{5,1}$, $A_{11,1}$, $C_{5,1}$, $D_{4,1}$, $D_{7,2}$, $E_{6,2}$ and $E_{7,3}$.
It follows from $\dim \tilde{V}_1=168$ that  $\tilde{V}_1$ has the type $A_{5,1}^4D_{4,1}$, $D_{4,1}^6$, $A_{5,1}E_{7,3}$ or $A_{5,1}C_{5,2}E_{6,2}$.
Since $\tilde{V}$ is a simple current extension of $V^{\sigma_h}$ graded by $\Z_3$,
there exists an order $3$ automorphism of $\tilde{V}_1$ of which the fixed-point subspace has the Lie algebra structure $A_{2,3}^6$.
Notice that a simple Lie algebra with non-inner order $3$ automorphisms is of type $D_4$.
One can easily see that the only possible type for $\tilde{V}_1$ is $D_{4,1}^6$.
By Proposition \ref{P:NVOA}, we have proved this proposition.
\end{proof}

\subsection{Main theorem for the Lie algebra $A_{2,3}^6$}
Set $N=Ni(D_4^6)$ and let $\sigma_2$ be the order $3$ automorphism of $V_N$ defined in Section \ref{S:D46}.
Then $\sigma_2$ satisfies Condition (I) in Section \ref{S:Orb} and the holomorphic VOA $(\widetilde{V_{N}})_{\sigma_2}$ has the weight one Lie algebra of type $A_{2,3}^6$ (\cite{SS}).
Note that the Lie algebra $(V_N^{\sigma_2})_1$ has also the type $A_{2,3}^6$.
We now prove the main theorem in this section.
\begin{theorem} Let $V$ be a strongly regular holomorphic VOA of central charge $24$ such that the Lie algebra structure of $V_1$ is $A_{2,3}^6$.
Then $V$ is isomorphic to the holomorphic VOA $(\widetilde{V_{N}})_{\sigma_2}$.
In particular, the structure of a strongly regular holomorphic vertex operator algebra of central charge $24$ is uniquely determined by its weight one Lie algebra if the Lie algebra has the type $A_{2,3}^6$.
\end{theorem}

\begin{proof} 
It suffices to verify the hypotheses in Section \ref{sec:4.2} for 
$\g=\mathfrak{p}=A_{2,3}^6$, $n=3$ and $W=V_{Ni(D_4^6)}$.
Take $h$ as in \eqref{Def:hA23} and set $\sigma=\sigma_h$.
Then by Lemma \ref{Lem:or3A23}, the order of $\sigma$ is $3$ on $V$.
The former hypotheses (a) and (b) hold by Lemma \ref{Lem:fixA23} and Propositions \ref{Prop:lowA23} and \ref{Prop:D46}.
The latter hypothesis about the uniqueness of the conjugacy class follows from Proposition \ref{Prop:ConjD4} (1).
\end{proof}

\section{Uniqueness of a holomorphic VOA with the Lie algebra $A_{5,3}D_{4,3}A_{1,1}^3$}
In this section, we consider a holomorphic VOA $V$ of central charge $24$ whose weight one Lie algebra has the type $A_{5,3}D_{4,3}A_{1,1}^3$. By applying the $\Z_3$-orbifold construction to $V$ and certain inner automorphism of order $3$, we show that one can obtain a holomorphic lattice VOA of central charge $24$ associated with the Niemeier lattice $Ni(D_4^6)$.
In addition, the structure of a strongly regular holomorphic VOA of central charge $24$ is uniquely determined by its weight one Lie algebra if the Lie algebra has the type $A_{5,3}D_{4,3}A_{1,1}^3$.

\subsection{Simple affine VOA of type $A_{1}$ at level $1$}\label{Sec:A11}
Let $\alpha_1$ be a simple root of type $A_1$ such that $(\alpha_1|\alpha_1)=2$.
Then the fundamental weight is $\Lambda_1=\alpha_1/2$.
Let $L_\Fg(1,0)$ be the simple affine VOA associated with the simple Lie algebra $\Fg$ of type $A_1$ at level $1$.
There exist exactly $2$ (non-isomorphic) irreducible $L_\Fg(1,0)$-modules,  $L_\Fg(1,0)$ and $L_\Fg(1,\Lambda_1)$. 

\begin{lemma}\label{Lem:A11} 
The lowest $L(0)$-weights of irreducible $L_\Fg(1,0)$-modules are given as follows:
\begin{table}[bht]
\begin{tabular}{|c|c|}
\hline
Weight $\lambda$& Lowest $L(0)$-weight of $L_{\mathfrak{g}}(1,\lambda)$ \\ \hline
$0$ & $0$\\\hline
$\Lambda_1$& $1/4$ \\ \hline
\end{tabular}
\end{table}
\end{lemma}

\subsection{Simple affine VOA of type $A_{5}$ at level $3$}
Let $\alpha_1,\alpha_2,\dots,\alpha_5$ be simple roots of type $A_5$ such that $(\alpha_i|\alpha_j)=2\delta_{i,j}-\delta_{|i-j|,1}$ for $1\le i,j\le5$.
Let $\Lambda_1,\Lambda_2,\dots,\Lambda_5$ be the fundamental weights such that $(\Lambda_i|\alpha_j)=\delta_{i,j}$, $1\le i,j\le 5$.
Let $L_\Fg(3,0)$ be the simple affine VOA associated with the simple Lie algebra $\Fg$ of type $A_5$ at level $3$.
It is well-known (cf. \cite{FZ}) that there exist exactly $56$ (non-isomorphic) irreducible $L_\Fg(3,0)$-modules. They are irreducible highest weight modules associated with dominant integral weights.

Set 
\begin{equation}
\Lambda=\frac{2}{3}\Lambda_3=\frac{1}{3}(\alpha_1+2\alpha_2+3\alpha_3+2\alpha_4+\alpha_5).\label{Eq:hA53}
\end{equation}
Then $(\Lambda|\Lambda)=2/3$.
For a dominant integral weight $\lambda$ of $\mathfrak{g}$, set $$n_{\Lambda}(\lambda)=\min\{(\Lambda|\mu)\mid \mu\in\Pi(\lambda)\}.$$
Here $\Pi(\lambda)$ denotes the set of all weights of the irreducible $\mathfrak{g}$-module with the highest weight $\lambda$.
By Corollary \ref{Cor:LW} and $\tau\gamma(\Lambda_3)=-\Lambda_3$, we have $n_{\Lambda}(\lambda)=(\Lambda|\tau\gamma(\lambda))=-(\Lambda|\lambda)$.

\begin{lemma}\label{Lem:A53} The dominant integral weights, the lowest $L(0)$-weights of the associated $L_\Fg(3,0)$-modules and the values of $(\Lambda|\lambda)$ and $n_\Lambda(\lambda)$ are given as follows:
\begin{center}
\begin{longtable}{|c|c|c|c|c|}
\hline
Weight $\lambda$ & Lowest $L(0)$-weight of $L_{\mathfrak{g}}(3,\lambda)$  & $(\Lambda|\lambda)$ & $n_\Lambda(\lambda)$ \\
\hline \hline 
$0$ & $0$& $0$& $0$ \\\hline
$\Lambda_1$, $\Lambda_5$ &$35/108$&  $1/3$& $-1/3$   \\\hline
$\Lambda_2$, $\Lambda_4$ &$14/27$&  $2/3$& $-2/3$   \\ \hline
$\Lambda_3$ &$7/12$& $1$ & $-1$   \\ \hline
$2\Lambda_1$, $2\Lambda_5$&$20/27$ & $2/3$   &$-2/3$  \\ \hline
$2\Lambda_2$, $2\Lambda_4$&$32/27$ & $4/3$   &$-4/3$  \\ \hline
$2\Lambda_3$ &$4/3$& $2$ & $-2$   \\ \hline
$\Lambda_1+\Lambda_2$, $\Lambda_4+\Lambda_5$&$11/12$ & $1$ & $-1$   \\ \hline
$\Lambda_1+\Lambda_3$, $\Lambda_3+\Lambda_5$&$26/27$ & $4/3$   & $-4/3$ \\ \hline
$\Lambda_1+\Lambda_4$, $\Lambda_2+\Lambda_5$&$95/108$ & $1$ &  $-1$  \\ \hline
$\Lambda_1+\Lambda_5$&$2/3$ & $2/3$ & $-2/3$   \\ \hline
$\Lambda_2+\Lambda_3$, $\Lambda_3+\Lambda_4$&$131/108$ & $5/3$   & $-5/3$ \\ \hline
$\Lambda_2+\Lambda_4$&$10/9$ & $4/3$ & $-4/3$   \\ \hline
$3\Lambda_1$,$3\Lambda_5$&$5/4$ & $1$    & $-1$\\ \hline
$3\Lambda_2$,$3\Lambda_4$&$2$ & $2$    & $-2$\\ \hline
$3\Lambda_3$ &$9/4$& $3$ &  $-3$  \\ \hline
$2\Lambda_1+\Lambda_2$, $\Lambda_4+2\Lambda_5$& $38/27$& $4/3$    & $-4/3$\\ \hline
$2\Lambda_1+\Lambda_3$, $\Lambda_3+2\Lambda_5$& $155/108$& $5/3$    & $-5/3$\\ \hline
$2\Lambda_1+\Lambda_4$, $\Lambda_2+2\Lambda_5$& $4/3$& $4/3$    & $-4/3$\\ \hline
$2\Lambda_1+\Lambda_5$, $\Lambda_1+2\Lambda_5$& $119/108$& $1$   & $-1$ \\ \hline
$\Lambda_1+2\Lambda_2$, $2\Lambda_4+\Lambda_5$& $179/108$& $5/3$    & $-5/3$\\ \hline
$\Lambda_1+2\Lambda_3$, $2\Lambda_3+\Lambda_5$& $191/108$& $7/3$    & $-7/3$\\ \hline
$\Lambda_1+2\Lambda_4$, $2\Lambda_2+\Lambda_5$& $19/12$& $5/3$    & $-5/3$\\ \hline
$2\Lambda_2+\Lambda_3$, $\Lambda_3+2\Lambda_4$& $215/108$&$7/3$&$-7/3$\\\hline
$2\Lambda_2+\Lambda_4$, $\Lambda_2+2\Lambda_4$&$50/27$&$2$&$-2$\\\hline
$2\Lambda_3+\Lambda_4$, $\Lambda_2+2\Lambda_3$&$56/27$&$8/3$&$-8/3$\\\hline
$\Lambda_1+\Lambda_2+\Lambda_3$, $\Lambda_3+\Lambda_4+\Lambda_5$&$5/3$ & $2$    & $-2$\\ \hline
$\Lambda_1+\Lambda_2+\Lambda_4$, $\Lambda_2+\Lambda_4+\Lambda_5$&$167/108$ & $5/3$    & $-5/3$\\ \hline
$\Lambda_1+\Lambda_2+\Lambda_5$, $\Lambda_1+\Lambda_4+\Lambda_5$&$35/27$ & $4/3$    & $-4/3$\\ \hline
$\Lambda_1+\Lambda_3+\Lambda_4$,$\Lambda_2+\Lambda_3+\Lambda_5$ &$44/27$& $2$ &    $-2$\\ \hline
$\Lambda_1+\Lambda_3+\Lambda_5$ &$49/36$& $5/3$   &  $-5/3$\\ \hline
$\Lambda_2+\Lambda_3+\Lambda_4$ &$23/12$& $7/3$   &  $-7/3$\\ 
\hline
\end{longtable}
\end{center}
\end{lemma}

\subsection{Simple affine VOA of type $D_{4}$ at level $3$}
Let $\alpha_1,\alpha_2,\alpha_3,\alpha_4$ be simple roots of type $D_4$ such that 
$(\alpha_i|\alpha_i)=2$ for $1\le i\le 4$, $(\alpha_2|\alpha_j)=-1$ and $(\alpha_i|\alpha_j)=0$ 
for $i, j\in \{1,3,4\}$.
Let $\Lambda_1,\Lambda_2,\Lambda_3,\Lambda_4$ be the fundamental weights such that $(\alpha_i|\Lambda_j)=\delta_{i,j}$ for $1\le i,j\le 4$.
Let $L_\Fg(3,0)$ be the simple affine VOA associated with the simple Lie algebra $\Fg$ of type $D_4$ at level $3$.
It is well-known (cf.\ \cite{FZ}) that there exist exactly $24$ (non-isomorphic) irreducible $L_\Fg(3,0)$-modules. They are irreducible  highest  weight modules associated with dominant integral weights. 

\begin{lemma}\label{Lem:D43} The dominant integral weights and the lowest $L(0)$-weights of the associated $L_\Fg(3,0)$-modules are given as follows:
\begin{center}
\begin{longtable}{|c|c|c|}
\hline
Weight $\lambda$ & Lowest $L(0)$-weight of $L_\mathfrak{g}(3,\lambda)$\\
\hline \hline
$0$&$0$\\\hline
$\Lambda_i$, $i\in\{1,3,4\}$&$7/18$\\\hline
$\Lambda_2$&$2/3$\\\hline
$2\Lambda_i$, $i\in\{1,3,4\}$&$8/9$\\\hline
$\Lambda_i+\Lambda_2$, $i\in\{1,3,4\}$&$7/6$\\\hline
$\Lambda_i+\Lambda_j$, $i,j\in\{1,3,4\}$, $i\neq j$&$5/6$\\\hline
$3\Lambda_i$, $i\in\{1,3,4\}$&$3/2$\\\hline
$\Lambda_1+\Lambda_3+\Lambda_4$&$4/3$\\\hline
$2\Lambda_i+\Lambda_j$, $i,j\in\{1,3,4\}$, $i\neq j$&$25/18$\\\hline
\end{longtable}
\end{center}
\end{lemma}

\subsection{Inner automorphism of a holomorphic VOA with Lie algebra $A_{5,3}D_{4,3}A_{1,1}^3$}\label{S:A53-2}

Let $V$ be a strongly regular holomorphic VOA of central charge $24$ whose weight one Lie algebra has the type $A_{5,3}D_{4,3}A_{1,1}^3$.
Let $V_1=\bigoplus_{i=1}^5\mathfrak{g}_i$ be the decomposition into the direct sum of simple ideals, where 
 the types of $\mathfrak{g}_1$, $\mathfrak{g}_2$ and $\mathfrak{g}_i$ $(3\le i\le 5)$ are $A_{5,3}$, $D_{4,3}$ and $A_{1,1}$, respectively.
Let $\mathfrak{H}$ be a Cartan subalgebra of $V_1$.
Then $\mathfrak{H}\cap\mathfrak{g}_i$ is a Cartan subalgebra of $\mathfrak{g}_i$.
Let $U$ be the subVOA generated by $V_1$.
Set $k_i=3$ for $i=1,2$ and $k_j=1$ for $j=3,4,5$.
By Proposition \ref{Prop:posl}, $U\cong \bigotimes_{i=1}^5 L_{\mathfrak{g}_i}(k_i,0)$.

Let 
\begin{equation}
h=
\frac{2}{3}(\Lambda_3,0,0,0,0)\in\bigoplus_{i=1}^5(\mathfrak{H}\cap\mathfrak{g}_i).\label{Def:hA53}
\end{equation}
One can easily check that 
\begin{align}\label{h3}
\langle h|h\rangle=3\times \left(\frac{2}{3}\Lambda_3\left|\frac{2}{3}\Lambda_3\right.\right)_{|\mathfrak{g}_1}=2
\end{align}
and that for all roots $\alpha$ of the semisimple Lie algebra $V_1$, we have 
\begin{align}\label{h3-1}
(h|\alpha)\ge-1.
\end{align}

\begin{lemma}\label{Lem:or3A53} The order of $\sigma_h$ is $3$ on $V$, and the $L(0)$-weights of the irreducible $\sigma_h$-twisted (resp. $\sigma_h^{-1}$-twisted) $V$-module $V^{(h)}$ (resp. $V^{(-h)}$) belong to $\Z/3$.
\end{lemma}
\begin{proof}
It follows from \eqref{Eq:hA53} that $(h|\lambda)\in\Z/3$ for any weights $\lambda$ of the semisimple Lie algebra $V_1$.
Hence $(\sigma_h)^3=1$.
In addition, $\sigma_h$ is an automorphism of order $3$ on $V_1$, so is it on $V$.
The latter assertion follows from Lemmas \ref{Lem:wtmodule} and \ref{Lem:A53} and \eqref{h3}.
\end{proof}

\begin{proposition}\label{Prop:lowA53} The lowest $L(0)$-weight of the irreducible $\sigma_h$-twisted (resp. $\sigma_{h}^{-1}$-twisted) $V$-module $V^{(h)}$ (resp. $V^{(-h)}$) is $1$.
In particular, $(V^{(\pm h)})_{i}=0$ for $i\in\{1/3,2/3\}$.
\end{proposition}
\begin{proof} Let $M\cong \bigotimes_{i=1}^5L_{\mathfrak{g}_i}(k_i,\lambda_i)$ be an irreducible $U$-submodule of $V$.
Let $\ell$ and $\ell^{(h)}$ be the lowest $L(0)$-weights of $M$ and of $M^{(h)}$, respectively.
By \eqref{h3-1}, we can apply Lemma \ref{Lem:lowestwt} to our case.
By \eqref{Eq:twisttop}, \eqref{Def:hA53} and \eqref{h3}, we have
\begin{align*}
\ell^{(h)}=\ell+n_{\Lambda}(\lambda_1)+1,
\end{align*}
where $\Lambda=(2/3)\Lambda_3$ in $\mathfrak{g}_1$.
If $\lambda_i=0$ for $1\le i\le 5$, then $\ell=0$ and $\ell^{(h)}=1$.
Assume $\lambda_i\neq0$ for some $1\le i\le 5$.
Then $\ell\ge2$.
Hence if $n_{\Lambda}(\lambda_1)\ge-2$, then $\ell^{(h)}\ge1$.
By Lemma \ref{Lem:A53}, $n_{\Lambda}(\lambda_1)<-2$ if and only if 
\begin{equation}\label{lambdaA5}
\lambda_1\in\{3\Lambda_3,\Lambda_1+2\Lambda_3,2\Lambda_3+\Lambda_5,2\Lambda_2+\Lambda_3,\Lambda_3+2\Lambda_4,2\Lambda_3+\Lambda_4,\Lambda_2+2\Lambda_3,\Lambda_2+\Lambda_3+\Lambda_4\},
\end{equation}
and $-3\le n_{\Lambda}(\lambda_1)<-2$ for every $\lambda_1$ in \eqref{lambdaA5}.
Clearly, $\ell$ is integral.
Recall that $\ell$ is the sum of lowest $L(0)$-weights of $L_{\mathfrak{g}_i}(k_i,\lambda_i)$, $1\le i\le 5$.
By Lemma \ref{Lem:A53}, if $$\lambda_1\in\{3\Lambda_3,2\Lambda_2+\Lambda_3,\Lambda_3+2\Lambda_4,2\Lambda_3+\Lambda_4,\Lambda_2+2\Lambda_3\},$$
then the lowest $L(0)$-weight of $L_{\mathfrak{g}_1}(\lambda_1)$ is greater than $2$, and hence $\ell\ge3$.
For the other $\lambda$,  there are no weights $\lambda_i$ $(2\le i\le 5)$ such that the lowest $L(0)$-weight 
of $\bigotimes_{i=1}^5L_{\mathfrak{g}_i}(k_i,\lambda_i)$ is $2$ by Lemmas \ref{Lem:A11}, \ref{Lem:A53} and 
\ref{Lem:D43}.
Hence, for $\lambda_1$ in \eqref{lambdaA5}, we have $\ell\ge3$.
Thus $\ell^{(h)}\ge1$.

One can also prove $\ell^{(-h)}\ge1$ by the same argument.
\end{proof}

\subsection{Identification of the Lie algebra: Case $D_{4,1}^6$}\label{S:A53-3}
Thanks to Proposition \ref{Prop:lowA53}, we can apply Theorem \ref{Thm:EMS}  to $V$ and 
$\sigma_h$, and obtain the strongly regular 
holomorphic VOA $\tilde{V}=\tilde{V}_{\sigma_h}$ of central charge $24$.
By the definition of $h$, we obtain the following:

\begin{proposition}\label{Prop:fixedA53} 
The Lie algebra structure of $(V^{\sigma_h})_1$ is $A_{2,3}^2U(1)D_{4,3}A_{1,1}^3$ and $\dim(V^{\sigma_h})_1=54$.
\end{proposition}

Let us identify the Lie algebra structure of  $\tilde{V}_1$.

\begin{proposition}\label{Prop:D462} 
The Lie algebra structure of $\tilde{V}_1$ is $D_{4,1}^6$.
In particular, $\tilde{V}$ is isomorphic to the lattice VOA associated with the Niemeier lattice $Ni(D_{4}^6)$.
\end{proposition}
\begin{proof} 
By Proposition \ref{Prop:V1}, the Lie algebra $\tilde{V}_1$ is semisimple.
By Proposition \ref{Prop:fixedA53}, we have $\dim (V^{\sigma_h})_1=54$, and 
by Proposition \ref{Prop:lowA53}, we have $(V^{(\pm h)})_{i}=0$ for $i=1/3,2/3$.
By Theorem \ref{Thm:Dimformula}, $$\dim\tilde{V}_1=4\times \dim (V^{\sigma_h})_1-\dim V_1+24=168;$$
hence we obtain the ratio $h^\vee/k=6$ by Proposition \ref{Prop:V1}.
By the same argument as in the proof of Proposition \ref{Prop:D46}, $\tilde{V}_1$ is of the type $A_{5,1}^4D_{4,1}$, $D_{4,1}^6$, $A_{5,1}E_{7,3}$ or $A_{5,1}C_{5,2}E_{6,2}$.
Since $\tilde{V}$ is a simple current extension of $V^{\sigma_h}$ graded by $\Z_3$,
there exists an order $3$ automorphism of the semisimple Lie algebra $\tilde{V}_1$ such that the fixed-point subspace has the Lie algebra structure $A_{2,3}^2D_{4,3}U(1)A_{1,1}^3$.
One can easily see that the only possible type of $\tilde{V}_1$ is $D_{4,1}^6$.
By Proposition \ref{P:NVOA}, we have proved this proposition.
\end{proof}

\subsection{Main theorem for the Lie algebra $A_{5,3}D_{4,3}A_{1,1}^3$}
Set $N=Ni(D_4^6)$ and let $\sigma_4$ be the order $3$ automorphism of $V_N$ defined in Section \ref{S:D46}.
Then $\sigma_4$ satisfies Condition (I) in Section \ref{S:Orb} and the holomorphic VOA $(\widetilde{V_{N}})_{\sigma_4}$ has the weight one Lie algebra of type $A_{5,3}D_{4,3}A_{1,1}^3$ (\cite{SS}).
Note that the Lie algebra $(V_N^{\sigma_4})_1$ has the type $A_{2,3}^2U(1)D_{4,3}A_{1,1}^3$.
We now prove the main theorem in this section.

\begin{theorem} Let $V$ be strongly regular holomorphic VOAs of central charge $24$ such that the Lie algebra structure of $V_1$ is $A_{5,3}D_{4,3}A_{1,1}^3$.
Then $V$ is isomorphic to the holomorphic VOA $(\widetilde{V_N})_{\sigma_4}$.
In particular, the structure of a strongly regular holomorphic VOA of central charge $24$ is uniquely determined by its weight one Lie algebra if the Lie algebra has the type $A_{5,3}D_{4,3}A_{1,1}^3$.
\end{theorem}

\begin{proof} 
It suffices to verify the hypotheses in Section \ref{sec:4.2} for 
$\g=A_{5,3}D_{4,3}A_{1,1}^3$, $\mathfrak{p}=A_{2,3}^2U(1)D_{4,3}A_{1,1}^3$, $n=3$ and 
$W=V_{Ni(D_4^6)}$.
Take $h$ as in \eqref{Def:hA53} and set $\sigma=\sigma_h$.
Then the order of $\sigma$ is $3$ by Lemma \ref{Lem:or3A53} and the former hypotheses (a) and (b) hold by Propositions \ref{Prop:lowA53}, \ref{Prop:fixedA53} and \ref{Prop:D462}.
The latter hypothesis about the uniqueness of the conjugacy class follows from Proposition \ref{Prop:ConjD4} (2).
\end{proof}

\paragraph{\bf Acknowledgement.} The authors wish to thank Nils Scheithauer for pointing out some gap in early version of this article.
They also wish to thank Kazuya Kawasetsu for helpful comments and the referees for useful suggestions.

\end{document}